\documentclass[11pt,a4paper]{article}
\usepackage[utf8]{inputenc}
\usepackage{amssymb}
\usepackage{amsthm}
\usepackage{amsmath}
\usepackage[square,sort&compress,numbers]{natbib}
\usepackage{blkarray}
\usepackage{comment}
\usepackage{paralist}
\usepackage{hyphenat}
\usepackage{wrapfig} 
\usepackage{tikz}
\usepackage{mathtools}
\usepackage{hyperref}
\hypersetup{
    colorlinks=true,
    linkcolor=blue,
    filecolor=magenta,      
    urlcolor=cyan,
}
\usepackage{centernot}
\usepackage[]{xcolor}
\usepackage{subcaption}
\usepackage[normalem]{ulem}
\usepackage{authblk}

\newtheorem{theorem}{Theorem}[section]
\newtheorem{notation}[theorem]{Notation}
\newtheorem{corollary}[theorem]{Corollary}
\newtheorem{lemma}[theorem]{Lemma}
\newtheorem{proposition}[theorem]{Proposition}
\newtheorem{conjecture}[theorem]{Conjecture}
\theoremstyle{definition}
\newtheorem{definition}[theorem]{Definition}
\newtheorem{remark}[theorem]{Remark}
\newtheorem{example}[theorem]{Example} 
\newtheorem{question}[theorem]{Question}

\newcommand{\supp}{\operatorname{supp}}
\newcommand{\F}{\mathbb{F}}
\newcommand{\N}{\mathbb{N}}
\newcommand{\C}{\mathcal{C}}

\author[1]{Ragnar Freij-Hollanti}
\author[2]{Relinde Jurrius}
\author[1]{Olga Kuznetsova}

\affil[1]{Department of Mathematics and Systems Analysis, Aalto University, Finland}
\affil[2]{Netherlands Defence Academy, the Netherlands}
\title{Combinatorial derived matroid}
\date{\vspace{-5ex}}
\hypersetup{final}
\begin{document}
\maketitle

\begin{abstract}
Let $M$ be an arbitrary matroid with circuits $\mathcal{C}(M)$. We propose a definition of a derived matroid $\delta M$ that has as its ground set $\mathcal{C}(M)$. Unlike previous attempts of such a definition, our definition applies to arbitrary matroids, and is completely combinatorial. We prove that the rank of $\delta M$ is bounded from above by $|M|-r(M)$, that it is connected if and only if $M$ is connected. We compute examples including the derived matroids of uniform matroids, the V\'amos matroid and the graphical matroid $M(K_4)$. We formulate conjectures relating our construction to previous definitions of derived matroids.
\end{abstract}

\section{Introduction}
\label{sec: intro}

Let $M$ be an arbitrary matroid. Gian-Carlo Rota~\cite{Rota} and Henry Crapo~\cite{Crapo,CrapoBlog} asked, partly independently and with various precise formulations, for a natural definition of a matroid $\delta M$ that has as its ground set the collection of (co)circuits of $M$. Programmatically, this would mean --- among other things --- a combinatorialization of notions such as syzygies and resolutions in algebraic geometry. For represented matroids, such a construction is known~\cite{oxley2019dependencies, jurrius2015coset}, but the combinatorics of $\delta M$ is in general not uniquely defined by that of $M$, unless $M$ is projectively unique~\cite{Oxley1996}, for example binary. The special case of binary matroids $M$ was carefully studied by Longyear~\cite{longyear1980circuit}. Another related notion is the lattice-theoretic notion of an adjoint geometric lattice to the lattice of flats of $M$~\cite{cheung}, but this does also not exist for all matroids. Moreover, when an adjoint geometric lattice exists, it is not necessarily unique~\cite{Alfter}. We propose a purely combinatorial construction of $\delta M$, defined via the rank function on $M$, and via an operation that resembles a closure operation on the collection of dependent sets. 

To our knowledge, the first actual definition in the literature of a derived matroid is by Longyear~\cite{longyear1980circuit}, in the case of binary matroids. However, these definitions are predated by predictions of their basic properties, made by Crapo~\cite{Crapo}, and more specific questions asked by Rota~\cite{Rota}. These predictions and questions, as well as some later work in the field, were formulated in two dual versions, with the ground set of the derived matroid being either the circuits or the cocircuits of the underlying matroid. For consistency, in this paper we reformulate all such results, whenever feasible, to the circuit-based definitions. 

We start by introducing the background on matroids and earlier work on derived matroids in Sections~\ref{sec:prelims} and~\ref{sec: earlier definitions}, respectively. In Section~\ref{sec:combinatorial derived matroid}, we introduce our construction and give several cryptomorphic descriptions of it. We then give several examples of combinatorial derived matroids in Section~\ref{sec: examples}. In Section~\ref{sec: connectedness} we show that the combinatorial derived matroid is connected if and only if the matroid is connected. In Section~\ref{sec: rank} we investigate the independent sets and rank of our construction, and in particular show that the rank is bounded from above by $|M|-r(M)$.  Section~\ref{sec: comparison of def} is devoted to comparing our construction to the earlier definitions discussed in Section~\ref{sec: earlier definitions}, and in particular deriving some conditions under which they are equal. 
Finally, in Section~\ref{sec: further research} we point out some directions for further research. 

\section{Preliminaries on matroids}
\label{sec:prelims}
In this paper, we use several cryptomorphic definitions of matroids. Our construction of the derived matroid will start with a matroid defined in terms of circuits, and construct the derived matroid in terms of the dependent sets of the latter.
We refer the reader to \cite{oxley2011matroid} for a  reference on matroid theory.
\begin{definition}[Matroid]\label{definition: matroid}
A matroid $M$ is a pair $(E,\mathcal{D})$ where $E$ is a finite set and $\mathcal{D}$ a family of subsets of $E$ satisfying
\begin{itemize}
\item[(D1)] $\emptyset\notin\mathcal{D}$;
\item[(D2)] if $D\in\mathcal{D}$ and $D\subseteq D'$ then $D'\in\mathcal{D}$;
\item[(D3)] if $D_1,D_2\in\mathcal{D}$ and $D_1\cap D_2\notin\mathcal{D}$, then $(D_1\cup D_2)\backslash\{e\}\in\mathcal{D}$ for all $e\in D_1\cap D_2$.
\end{itemize}
We call the members of $\mathcal{D}$ the \emph{dependent sets} of $M$. The subsets of $E$ that are not in $\mathcal{D}$ are thus {\em independent}. An inclusion-minimal dependent set is called a \emph{circuit}, and the family $\mathcal{C}$ of circuits of a matroid $M$ is characterized by the axioms
\begin{itemize}
\item[(C1)] $\emptyset\notin\mathcal{C}$;
\item[(C2)] if $C_1,C_2\in\mathcal{C}$ and $C_1\subseteq C_2$ then $C_1=C_2$;
\item[(C3)] if $C_1,C_2\in\mathcal{C}$ are distinct and $e \in C_1\cap C_2$, then there exists $C_3 \in \mathcal{C}$ such that $C_3\subseteq(C_1\cup C_2)\backslash\{e\}$.
\end{itemize}
\end{definition}

\begin{notation} 
Unless explicitly mentioned otherwise, $E$ denotes a finite set, and its power set is denoted by $2^E$. The set of natural numbers $\N$ includes zero. If $A$ and $B$ are finite sets and $\F$ is a field, then $\F^{A\times B}$ is the set of matrices with rows indexed by $A$ and columns indexed by $B$. We write $\F^{m\times B}$ as a shorthand for $\F^{\{1,2,\dots, m\}\times B}$, and similarly $\F^{A\times n}$ as a shorthand for $\F^{A\times \{1,2,\dots, n\}}$.
\end{notation}

\begin{definition}[Rank and nullity]\label{definition: rank}
Let $M=(E,\mathcal{D})$ be a matroid on the finite ground set $E$. The {\em rank function} $r:2^E\to \N$ is defined by $$r(S)=\max\{|T|: T\subseteq S, S\not\in\mathcal{D}\},$$ and the \emph{nullity function} $n:2^E\to \N$ is defined by $n(S)=|S|-r(S)$.
\end{definition}

It follows from property (D2) that rank and nullity are both increasing functions on $2^E$, and from (D3) that the rank function is submodular and the nullity function is supermodular, meaning that $$r(S)+r(T)\geq r(S\cap T)+r(S\cup T)$$ and $$n(S)+n(T)\leq n(S\cap T)+n(S\cup T)$$ for all $S, T\subseteq E$~\cite{oxley2011matroid}.

Let $R$ be a matrix with entries in the field $\mathbb{F}$. The matroid associated to $R$ is a matroid $M(R)$ whose ground set is the set of columns of $R$ and a subset of columns is said to independent if the corresponding submatrix is full-rank. A matroid $M$ is called \emph{representable} over the field $\mathbb{F}$ if there exists some $R \in \mathbb{F}^{m \times E}$ such that $M(R)=M$. We sometimes refer to the row space of $R$ as a code representing $M$.

We will also need the notion of {\em fundamental circuits} with respect to a fixed basis $B$ of $M$.
\begin{definition}\label{def:fund_circuits}
Let $B$ be a basis for $M$. For $e\in E\setminus B$, the \emph{fundamental circuit} $C_{eB}$ is the unique circuit contained in $B\cup \{e\}$.
 \end{definition}
Uniqueness follows as $n(B\cup \{e\})=1$, and since $B$ is independent it follows that $e\in C_{eB}$.

\section{Earlier definitions of derived matroid}\label{sec: earlier definitions}

Several authors have investigated structures similar to the combinatorial derived matroid we are about to define. As explained before, the idea goes back to Crapo and Rota. We will describe four approaches to the question and explain how these are related. 

\subsection{Longyear}

Longyear~\cite{longyear1980circuit} introduced the notion of the derived matroid when $M$ is a binary matroid. The primary motivation in~\cite{longyear1980circuit} was to display the essential structure of a binary matroid in a computationally more tractable way. To do so, Longyear defined  the notion of circuit bases, and correspondingly dependent sets of circuits, in terms of {\em Kirkhoff sums}, also known as iterated symmetric difference, and equivalent to binary sums. The Kirkhoff sum of the sets $A_1,\dots A_k$ is the set $$A_1\oplus\dots\oplus A_k=\left\{x\in\bigcup_{i=1}^k A_i : \# \{i: x\in A_i\}\mbox{ is odd}\right\}.$$ 

\begin{definition}
A circuit basis in a matroid $M$ is a minimal collection $A\subseteq\mathcal{C}(M)$ such that every circuit in $\mathcal{C}(M)$ is the Kirkhoff sum of some subset of $A$.
\end{definition}

It follows from the definition that a set $A \subseteq \mathcal{C}$ is dependent in $\delta M$ precisely if it has a non-empty subset $A' \subseteq A$ such that the Kirkhoff sum of the circuits in $A'$ is empty. In other words, $A'\subseteq A$ should be such that every element in $E$ is be contained in an even number of the sets in $A'$. Longyear showed that under this definition, the derived matroid of a binary matroid is also a binary matroid. However, for a non-binary matroid, the set of circuit bases have no reason to constitute the set of basis for a matroid.

\begin{definition}
Let $M$ be a binary matroid. The matroid whose ground set is $\mathcal{C}(M)$ and whose bases are the circuit bases in $M$ is called the \emph{Longyear derived matroid}, and is denoted $\delta_{L}(M)$.
\end{definition}

\subsection{The Oxley--Wang definition}\label{sec: Oxley Wang construction}

A more general definition was introduced in~\cite{oxley2019dependencies} for represented matroids, which
are matroids along with a fixed representation. The derived matroid $\delta M(R)$ of a representation $R$  of the matroid $M$ is determined by the generator matrix whose column vectors are the minimal support vectors in the dual of the given representation of $M$. By construction, such derived matroids are representable.

Let $M=(E,\mathcal{C})$ be a  matroid on the finite ground set $E$ with the set of circuits $\mathcal{C}$. To motivate our further definitions, we first consider the case of a matroid $M$ represented by a code $Q\subseteq \F^n$, with dual code $Q^\perp\subseteq \F^n$ for some field $\F$. As is commonplace in coding theory, for $S\subseteq[n]$, we denote by $Q|_S$ the projection of $Q$ to the coordinates in $S$, and by $Q(S)$ the {\em shortened code} $$Q(S)=\{q\in Q :\supp(q)\subseteq S\}.$$ For a circuit $C\in\mathcal{C}(M)$, there is a codeword $q_C$ in $Q^\perp$ with $\supp(q_C)=C$, and such a codeword is unique up to scalar multiple. We call $q_C$ a {\em circuit vector} of $Q$ on $C$. 
\begin{definition}\label{def:derived_code}
Let $Q\subseteq \F^n$ be a linear code, and let $\C$ be the collection of circuits of the associated matroid. Moreover, let $A\in\F^{n\times\C}$ be the matrix that has a circuit vector $q_C$ in the column indexed by the circuit $C\in\C$. We define the {\em derived code} $\delta(Q)$ to be the row span of $A$, and the {\em Oxley--Wang derived matroid} $\delta_{OW}(Q)$ to be the matroid represented by $\delta(Q)$. \end{definition}

Clearly, if $\mathbf{v}\in Q^\perp$ is not a circuit vector, then $\mathbf{v}$ can be written as the sum of two vectors in $Q^\perp$ whose support is strictly contained in $\supp \mathbf{v}$. Thus, $Q^\perp$ is generated by its minimal support vectors, {i.e.} the circuit vectors. It immediately follows that the column space of the matrix $A$ in Definition~\ref{def:derived_code} is $Q^\perp$, and so the rank of $\delta_{OW}$ is $\dim\delta(Q)=r(A)=\dim(Q^\perp)=n-\dim(Q)$.

We can also present a family of explicit bases for $\delta_{OW}$ via fundamental circuits. 
 For a fixed basis $B$ and distinct elements $e,f\not\in B$, we clearly have $e\not\in C_{fB}$. 
 It follows that the collection $$\delta B=\left\{C_{eB} : e\in E\setminus B\right\}$$ is independent in $\delta_{OW}$. Since $|\delta B|=|E\setminus B|=n-\dim Q$, we see that $\delta B$ is a basis for $\delta_{OW}(Q)$ for every basis $B$ of $M(Q)$.

\subsection{Jurrius--Pellikaan}\label{sec:derived code JP}

Independently of the work in~\cite{longyear1980circuit}, Jurrius and Pellikaan~\cite{jurrius2015coset} defined the derived code of a linear code. This work was then not noticed by Oxley and Wang~\cite{oxley2019dependencies}. Both constructions give a derived matroid for a representable matroid, and we will see that the definitions are dual to each other. The main difference is thus that the following definition from~\cite{jurrius2015coset} uses the cocircuits as the ground set for the derived code. The construction is based on the construction of a derived hyperplane arrangement given in \cite[\S 5.10]{crapo:2009}.

\begin{definition}[\cite{jurrius2015coset}, Definition 4.1]
Let $G$ be a $k \times n$ matrix of rank $k$ over a field $\mathbb{F}$. Let $J\subseteq \{ 1, \ldots ,n\}$. Let $\mathbf{g} _j$ be the $j$-th column of $G$. Let $G_J$ be the $k \times |J|$ submatrix of $G$ consisting of the columns $\mathbf{g}_j$ with $j$ in $J$. Let $G_J(X)$ be the $k \times (|J|+1)$ matrix obtained by adding the column $X=(X_1, \ldots ,X_k)^T$ to $G_J$, where $X_1, \ldots ,X_k$ are variables. Let $G_{i,J}$ be the $(k-1) \times |J|$ matrix obtained by deleting the $i$-th row of $G_J$. Let $\Delta_{J}(X)= \det (G_J(X))$ and $\Delta_{i,J}= \det (G_{i,J})$  in case $|J|=k-1$. Let $D(G)$ be the {\em derived matrix} of $G$ of size $k \times \binom{n}{k-1}$ with entries $\Delta_{i,J}$ with $i=1,\ldots ,k$ and  $J\subseteq \{ 1, \ldots ,n\}$ of size $k-1$ ordered lexicographically. The matrix $D_2(G)$ is obtained by removing zero columns and columns that are a scalar multiple of another column. The {\em derived code} $D_2(C)$ is the code with generator matrix $D_2(G)$.
\end{definition}

It was shown in \cite{jurrius2015coset} that the construction of $D_2(C)$ is independent of the choice of generator matrix $G$.
Geometrically, $D_2(G)$ is constructed as follows. For every subset of $k-1$ columns of $G$, check if they are independent in $\mathbb{F}^k$. If not, add a zero column to $D(G)$. If they are independent, take the coefficients of a linear form determining this $(k-1)$-space and make them into a column of $D(G)$. Then remove zeros and multiple subsets that span the same $(k-1)$-space.

From this geometric interpretation, one can see that this construction is the same as $\delta_{OW}$ of the dual matroid. Every set of $k-1$ columns of $G$ (i.e., elements of $M$) that has rank $k-1$ contributes to a column in $D(G)$. If another set of $k-1$ columns spans the same $(k-1)$-space, the corresponding column in $D(G)$ will be removed from $D_2(G)$. Hence all the columns of $D_2(G)$ correspond to a hyperplane of the matroid represented by $G$. Since there is a one-to-one correspondence between hyperplanes of $M$ and circuits of $M^*$, this gives the desired duality relation. Note that in the construction of $\delta_{OW}$ one gets a representation that is not of full rank, whereas the definition of $D_2(G)$ does, by definition, give a full rank matrix.

\subsection{Adjoint matroid}\label{sec:adjoint}

The construction of an adjoint of a matroid is closely related to that of the derived code. Where the derived code always exists, the adjoint does not. It was noted in ~\cite[Remark 4.12]{jurrius2015coset} that the construction of the derived code could not be generalised to matroids.

Let $L$ be a geometric lattice. That is, the lattice of flats of a simple matroid. Intuitively, we construct the adjoint of $L$ as follows: take the diagram of $L$, flip it upside down, and add some elements to make this lattice geometric again. Flipping the lattice upside down gives a lattice with rank function $n(A)=r(L)-r(A)$, the nullity function of the matroid. This function satisfies the first two of the rank axioms for matroids~\cite{oxley2011matroid}, but it is not submodular. So, we need to add new elements for every pair $A,B\in L$ such that $n(A\cup B)+n(A\cap B)>n(A)+n(B)$. This is the case for every pair that is not modular in $L$.

We will now give a formal definition of an adjoint~\cite{cheung}.

\begin{definition}\label{definition: adjoint}
Let $L$ be a geometric lattice. Let $L^{opp}$ be the opposite lattice of $L$, that is, the lattice on the same elements with the order relation reversed. We say that $L^\Delta$ is an \emph{adjoint} of $L$ if $L^\Delta$ is a geometric lattice such that $r(L^\Delta)=r(L^{opp})$, there exists an injective embedding $e:L^{opp}\hookrightarrow L^\Delta$, and the points of $L^{opp}$ are in bijection with the points of $L^\Delta$. We write $\psi:L\to L^\Delta$ for the map that fist takes $L$ to $L^{opp}$ and then applies $e$.
\end{definition}

This seemingly not-too-demanding definition has some strict implications~\cite{cheung}:

\begin{lemma}\label{lemma: adjoint properties}
The map $\psi:L\to L^\Delta$ has the following properties:
\begin{enumerate}
\item If $A$ covers $B$ in $L$, then $\psi(B)$ covers $\psi(A)$ in $L^\Delta$.
\item For all $A\in L$, $r_{L^\Delta}(\psi(A))=r(L)-r(A)$.
\item For $A,B\in L$ we have that $\psi(A\vee B)=\psi(A)\wedge\psi(B)$.
\end{enumerate}
\end{lemma}

The construction of the derived code by Jurrius and Pellikaan gives in fact an adjoint: the map $\psi$ defined in \cite[Definition 4.15]{jurrius2015coset} satisfies the properties of the map $\psi$ in Definition \ref{definition: adjoint}.

\section{Construction of the combinatorial derived matroid}
\label{sec:combinatorial derived matroid}

We are now ready to present our definition of the combinatorial derived matroid $\delta M$ of a matroid $M$ with circuit set $\mathcal{C}$. In fact, we will make two such definitions, one of which defines the circuits and the other of which defines the dependent sets of $\delta M$. The definitions are very similar, but a priori different, and we will introduce them both before we prove that they indeed yield the same matroid.

For motivation, we first assume that $M$ is represented by some code $Q$ over a field $\F$. Like in Section~\ref{sec: Oxley Wang construction}, for each set $S$ of $M$ we define $$Q^\perp(S)=\{q\in Q^\perp :\supp(q)\subseteq S\},$$ and for each circuit $C$ of $M$ we let $q_C$ be a vector in $Q^\perp$ with $\supp(q_C)=C$, noting that $q_C$ is unique up to scalar multiple. 
For any collection $A\subseteq\mathcal{C}$ of circuits, we have $$\operatorname{span}\{ q_C : C\in A\}\subseteq Q^\perp(\cup_{C\in A} C),$$ where $\operatorname{span}$ denotes the linear span over $\F$. Hence, if \begin{equation}\label{eq:dimcond1}
    \dim\left(Q^\perp(\cup_{C\in A} C)\right)<|A|,
\end{equation} then the dual codewords $q_C :C\in A$ are linearly dependent regardless of the chosen representation $Q$, and so we would like them to be dependent in the combinatorial derived matroid. 

Now by standard duality arguments, $$Q^\perp(S)=(Q|_S)^\perp,$$ where $Q|_S$ is the projection of $Q$ to the coordinates in $S$. Thus $$\dim(Q^\perp(S))=|S|-\dim(Q|_S)=n(S),$$ so the condition \eqref{eq:dimcond1}, which is sufficient for dependence of $\{q_C : C\in A\}$ can be rewritten as 
\begin{equation}\label{eq:dimcond2}
    n(\cup_{C\in A} C)<|A|.
\end{equation}
For an arbitrary, not necessarily representable matroid, we build the definition of dependence of $A\subseteq \mathcal{C}$ in the combinatorial derived matroid on condition \eqref{eq:dimcond2}. First we define two operations on sets of circuits that we will later use to construct the dependent sets of the combinatorial derived matroid.

\begin{definition}\label{def:Operations}
Let $\mathcal{C}$ be the set of circuits of some matroid, and let $\mathcal{A}\subseteq \mathcal{C}$. Then we define the collection
$$\epsilon(\mathcal{A})=\mathcal{A}\cup\left\{(A_1 \cup A_2) \setminus \{C\} : A_1, A_2\in \mathcal{A}, A_1\cap A_2\not\in\mathcal{A}, C\in A_1\cap A_2\right\}.$$ Moreover, we define $${\uparrow}\mathcal{A}=\{A\subseteq \mathcal{C}: \exists A'\in\mathcal{A}: A'\subseteq A\},$$ 
and denote $\min\mathcal{A}\subseteq\mathcal{C}$ the collection of inclusion minimal sets in $\mathcal{A}$.
\end{definition}

 For any matroid with ground set $\mathcal{C}$, circuit set $\mathcal{B}$ and collection of dependent sets $\mathcal{A}$, we must have $\min\mathcal{A}=\mathcal{B}$ and ${\uparrow}\mathcal{B}=\mathcal{A}$.  Observe that, by definition, $\mathcal{A}\subseteq {{\uparrow} \mathcal{A}}$ and $\mathcal{A}\subseteq \epsilon(\mathcal{A})$ for every $\mathcal{A}\subseteq2^\mathcal{C}$. 
The operation $\epsilon: 2^\mathcal{C}\to 2^\mathcal{C}$ is designed to guarantee properties (C3) and (D3) in the matroid axioms.

\begin{definition}\label{def:combinatorial derived matroid sets} Let $M$ be a matroid, and $\mathcal{C}=\mathcal{C}(M)$ its collection of circuits. Define the collection \begin{equation}\label{eq:A0 definition}
  \mathcal{A}_0:=\{A \subseteq \mathcal{C}: |A|> n(\cup_{C\in A} C)\}.  
\end{equation} Inductively, we let $\mathcal{A}_{i+1}={\uparrow}\epsilon(\mathcal{A}_i)$ for $i\geq 1$, and $$\mathcal{A}=\bigcup_{i\geq 0} \mathcal{A}_i.$$ Analogously, we let $\mathcal{B}_0=\min\mathcal{A}_0$, $\mathcal{B}_{i+1}=\epsilon(\mathcal{B}_i)$ for $i\geq 1$, and $$\mathcal{B}=\min\bigcup_{i\geq 0} \mathcal{B}_i.$$

For a set $A\in\mathcal{A}_{i+1}\setminus\mathcal{A}_i$, respectively $B\in\mathcal{B}_{i+1}\setminus\mathcal{B}_i$, we say that it has {\em depth} $i+1$ in $\mathcal{A}$ or $\mathcal{B}$ respectively.
\end{definition}

Note that the sequences $\mathcal{A}_i$ and $\mathcal{B}_i$ are both increasing and contained in the finite set $2^\mathcal{C}$. Hence, we have that $\mathcal{A}_0\subseteq\mathcal{A}$ and $\mathcal{A}=\mathcal{A}_n$ for some $n\geq 0$, and analogously $\mathcal{B}_0\subseteq\mathcal{B}$ and $\mathcal{B}=\mathcal{B}_n$ for some $n\geq 0$.
\begin{definition}\label{def:combinatorial derived matroid}
Let $M$ be a matroid with circuits $\mathcal{C}$. Then the \emph{combinatorial derived matroid} $\delta M$ is a matroid with ground set $\mathcal{C}(M)$ and dependent sets $\mathcal{A}$. Alternatively, we can define $\delta M$ by its circuits $\mathcal{B}$.
\end{definition}

We will prove, in Propositions~\ref{prop:dM_is_matroid} and \ref{proposition: do operation 2 once}, that these definitions gives indeed a matroid and that the two definitions are the same. After this, we will define yet another way to construct the circuit set $\mathcal{B}$, in Proposition~\ref{proposition: construction of circuits}.

\begin{remark}
 Rewriting axiom (D3) in the logically equivalent form:
 \begin{align*}
    &\mbox{if }D_1, D_2\in\mathcal{D},\\ &\mbox{then } D_1\cap D_2\in\mathcal{D} \mbox{ or } (D_1\cup D_2)\setminus\{e\}\in\mathcal{D} \mbox{ for all } e\in D_1\cap D_2,\end{align*}
we could in principle for every $A_1, A_2\in\mathcal{A}_i$ such that $A_1\cap A_2\not\in\mathcal{A}_i$, have chosen to include either $A_1\cap A_2$ or $(A_1\cup A_2)\setminus\{C\}$ into $\mathcal{A}_{i+1}$. By defining $\mathcal{A}_{i+1}$ via the operation $\epsilon$, as in Definition~\ref{def:combinatorial derived matroid sets}, we guarantee to not construct any {\bf small} dependent sets when not necessary. This appears as a desirable property for the constructed derived matroid to be ``generic''.
    Indeed, consequently choosing to include $D_1\cap D_2$ in $\mathcal{A}$ for all prescribed dependent sets $D_1, D_2\in\mathcal{A}$ would result in a rather less interesting matroid --- often even with rank $0$. 
\end{remark}

\begin{notation}
When $A\subseteq\mathcal{C}$, we write $n(A)$ for $n(\cup_{C\in A} C)$ in the matroid $M$ and $\operatorname{supp}(A)$ for $\bigcup_{C \in A}C$, which stands for the \emph{support} of $A$.
\end{notation}

\begin{lemma}\label{lem: cardinality nondecreasing}
Let $A\in\mathcal{A}_{i+1}$. Then there is an $A'\in\mathcal{A}_i$ such that $|A'|\leq|A|$.
\end{lemma}
\begin{proof}
This is clear if $A\in\mathcal{A}_{i}$, so assume it is not. Moreover, since the inclusion minimal sets in $\mathcal{A}_{i+1}={\uparrow}\epsilon(\mathcal{A}_i)$ are contained in $\epsilon(\mathcal{A}_i)$, we may assume that 
$A\in\epsilon(\mathcal{A}_i)\setminus\mathcal{A}_i$. Thus, there exist $A_1,A_2\in\mathcal{A}_i$ with $A_1\cap A_2\notin\mathcal{A}_i$ and $A=(A_1\cup A_2)\backslash\{C\}$ for some $C\in A_1\cap A_2$. 
Then $A_1\not\subseteq A_2$ since $A_1\in\mathcal{A}_i$ but $A_1\cap A_2\not\in\mathcal{A}_i$. Hence we have $|A_2|\leq|A_1\cup A_2|-1=|A|$, so the lemma holds with $A'=A_2$.
\end{proof}

\begin{proposition}\label{prop:dM_is_matroid}
For any matroid $M=(E,\mathcal{C})$ 
the collection $\mathcal{A}$ is the collection of dependent sets of some matroid with ground set $\mathcal{C}$.
\end{proposition}

\begin{proof}
We will confirm that $\delta M=(\mathcal{C},\mathcal{A})$ satisfies the axioms for dependent sets of a matroid as given in Definition~\ref{definition: matroid}.
\begin{enumerate}
    \item Since $|\emptyset|=0$ and nullity is non-negative, $\emptyset \notin \mathcal{A}_0$. By Lemma~\ref{lem: cardinality nondecreasing}, we inductively have $\emptyset\not\in\mathcal{A}_i$ for every $i\geq 0$, so $\emptyset\not\in\mathcal{A}$. 
    \item If $A_1 \in \mathcal{A}$ then there exists some minimal $i$ such that $A_1 \in \mathcal{A}_i$. Then for all $A_2 \subseteq \mathcal{C}$ such that $A_1 \subseteq A_2$, $A_2 \in \mathcal{A}_{i+1}$ and $\mathcal{A}_{i+1} \subseteq \mathcal{A}$ by construction.
    \item Let $A_1$ and $A_2$ be two distinct sets in $\mathcal{A}$. If for some $i$, $A_1 \cap A_2 \in \mathcal{A}_i$, then $A_1 \cap A_2 \in \mathcal{A}$ and there is nothing to check. Otherwise, $A_1 \cap A_2 \notin \mathcal{A}$ and there exists a minimal $i$ such that $A_1, A_2 \in \mathcal{A}_i$. Therefore, $\forall C \in A_1 \cap A_2$, $(A_1 \cup A_2) \backslash \{C\} \in \mathcal{A}_{i+1} \subseteq \mathcal{A}$. \qedhere
\end{enumerate}
\end{proof}

Next, we will study the circuits of the derived matroid $\delta M$ that we just defined. The collection of circuits is by definition $\min\mathcal{A}$, and we will show in Proposition~\ref{proposition: do operation 2 once} that this is equal to $\mathcal{B}$, as constructed in Definition~\ref{def:combinatorial derived matroid sets}.

\begin{lemma}[Circuits of $\delta M$]\label{lemma: derived circuits}  
Let $A\in\min\mathcal{A}$ have depth $i+1\geq 1$. Then 
there exist sets $A_1,A_2 \in \min\mathcal{A}$ of depth at most $i$, such that $A=(A_1 \cup A_2) \setminus \{C\}$ for some $C \in A_1 \cap A_2$. 
\end{lemma}

\begin{proof}
The assumption says that $A \in\mathcal{A}_{i+1}\setminus \mathcal{A}_i$. Since $A$ is minimal in $\mathcal{A}$, it must in particular be minimal in $\mathcal{A}_{i+1}={\uparrow}\epsilon\mathcal{A}_i$, so $A\in\epsilon\mathcal{A}_i$. Hence, $A=(A_1 \cup A_2) \setminus \{C\}$ for some $C \in A_1 \cap A_2$ and $A_1,A_2 \in \mathcal{A}_i\subseteq\mathcal{A}$. Note that there is a choice involved for $A_1$ and $A_2$: we pick them such that there is no $A_1'\subsetneq A_1$ with $A=(A_1' \cup A_2) \setminus \{C\}$ and no $A_2'\subsetneq A_2$ with $A=(A_1 \cup A_2') \setminus \{C\}$.

We now need to show that $A_1,A_2\in\min\mathcal{A}$. For contradiction, assume $A_1$ is not a circuit of $\delta M$. Then it contains a circuit $A_3 \subsetneq A_1$. If $A_3$ does not contain $C$, then $A_3 \subsetneq (A_1 \cup A_2) \backslash \{C\} =A$, which contradicts the assumption that $A$ is a circuit of $\delta M$. On the other hand, if $C\in A_3$, then $(A_3 \cap A_2)\subsetneq A_3\in\min\mathcal{A}$ so $(A_3 \cap A_2)\notin \mathcal{A}$. But then we would have $A\supsetneq(A_3 \cup A_2) \setminus \{C\} \in \mathcal{A}$ by the assumption on $A_1$, which again contradicts the assumption that $A\in\min\mathcal{A}$. The claim follows. 
\end{proof}

\begin{proposition}\label{proposition: do operation 2 once}
Let $M$ be a matroid and let $\mathcal{B}\subseteq 2^\mathcal{C}$ be constructed as in Definition~\ref{def:combinatorial derived matroid sets}. Then $\mathcal{B}$ is the collection of circuits of $\delta M$.
\end{proposition}

\begin{proof}
We need to show that $\mathcal{B}=\min\mathcal{A}$, and we will show by induction that $\min\mathcal{B}_i=\min\mathcal{A}_i$ for every $i$.

By definition, we have $\min\mathcal{A}_0=\mathcal{B}_0$. Now Lemma~\ref{lemma: derived circuits} shows that the sets in $\min\mathcal{A}_{i+1}\setminus\mathcal{A}_i$ are precisely those that can be obtained from sets in $\min\mathcal{A}_i$ via the $\epsilon$ operation. Since by definition we have $\mathcal{B}_{i+1}=\epsilon\mathcal{B}_i$, the claim follows by induction.
\end{proof}

From our previous results, it follows that we have the next construction of the circuits of the combinatorial derived matroid.

\begin{proposition}[Construction of circuits of $\delta M$]\label{proposition: construction of circuits}
The set of circuits of $\delta M$ can be constructed iteratively as follows:
\begin{enumerate}
    \item Let $\mathcal{E}_0=\min\mathcal{A}_0$.
    \item Let $\mathcal{E}_{i+1}=\min\epsilon\mathcal{E}_i$ for all $i\geq 0$. 
    \item The sequence $\mathcal{E}_i$ terminates, and its limit equals the collection $\mathcal{B}$ of circuits of $\delta M$.
   
\end{enumerate}
\end{proposition}
\begin{proof}
From Proposition~\ref{proposition: do operation 2 once}, it suffices to show that $\min\mathcal{B}_i=\mathcal{E}_i$ for every $i\geq 0$. This follows by induction, if we can show that every element of depth $i$ in $\min\mathcal{B}$ can be written as $B_1\cup B_2\setminus\{C\}$ for some {\em minimal} sets in $\mathcal{B}_{i-1}$.  
The proof is similar to that of Lemma~\ref{lemma: derived circuits}. Let $B$ be a minimal element of $\mathcal{B}_{i+1}$ that is not in $\mathcal{B}_i$. By construction, $B\in\epsilon\mathcal{B}_i$, so, $B=(B_1 \cup B_2) \backslash \{C\}$ for some $C \in B_1 \cap B_2$, $B_1 \cap B_2 \not\in \mathcal{B}_i$, and $B_1,B_2 \in \mathcal{B}_i$. Moreover, we can choose $B_1, B_2$ such that there is no $B_1'\subsetneq B_1$ with $B=(B_1' \cup B_2) \setminus \{C\}$ and no $B_2'\subsetneq B_2$ with $B=(B_1 \cup B_2') \setminus \{C\}$.

For contradiction, assume $B_1$ is not a minimal element in $\mathcal{B}_i$. Then it contains a minimal element $B_3 \subsetneq B_1$. If $B_3$ does not contain $C$, then $B_3 \subsetneq (B_1 \cup B_2) \backslash \{C\} =B$, which contradicts the assumption that $B$ is a minimal element in $\mathcal{B}_{i+1}$. Alternatively, if $B_3$ contains $C$, then $(B_3 \cap B_2)\subsetneq B_3$, so $B_3 \cap B_2\notin \mathcal{B}_i$ because $B_3$ is a minimal element in $\mathcal{B}_i$. Hence $(B_3 \cup B_2) \backslash \{C\} \in \mathcal{B}_{i+1}$. However, $(B_3 \cup B_2) \backslash \{C\}\subsetneq (B_1 \cup B_2) \backslash \{C\}=B$, by the assumption on $B_1$, which is again a contradiction. The claim follows.
\end{proof}

We conclude this section by showing that the combinatorial derived matroid is a simple matroid, so it has no circuits of size $1$ and $2$, but if $M$ is connected, it does have s a circuit of size $3$.

\begin{lemma}\label{lemma: delta M is simple}
Let $M$ be a matroid. Then $\delta M$ is simple, that is, there are no dependent sets of size $1$ or $2$.
\end{lemma}
\begin{proof}
By Lemma \ref{lem: cardinality nondecreasing} it suffices to show that there are no sets of size $1$ or $2$ in $\mathcal{A}_0$. If $A=\{C\}$, we have that $n(A)=n(C)=1$ hence $A\notin\mathcal{A}_0$. Now let $A=\{C_1,C_2\}$. Since $C_1$ and $C_2$ are not contained in one another, $C_1\cap C_2$ is independent in $M$ hence has nullity $0$. So, by supermodularity of the nullity function, we get that $n(A)=n(C_1\cup C_2)=n(C_1\cup C_2)+n(C_1\cap C_2)\geq n(C_1)+n(C_2)=2$ hence $A\notin\mathcal{A}_0$.
\end{proof}

\begin{proposition}[Triangles in $\delta M$]\label{proposition: delta M has triangle}
Suppose that $M$ is a connected matroid with at least two circuits. Then every element of $\delta M$ is contained a triangle, that is, a circuit of size $3$.
\end{proposition}

\begin{proof}
Let $C$ and $D$ be two circuits of $M$ with $C\cap D\neq\emptyset$. Then we have
\begin{align*}
n(C\cup D) 
 &=n(C\cup D)+n(C\cap D) \\
 &\geq n(C)+n(D) \\
 &=1+1=2.
\end{align*}
Let $F\subseteq C\cup D$ be inclusion minimal with $n(F)=2$. Then $F$ contains a circuit $C'$ other than $C$. Since $$C'\setminus C\subseteq F\setminus C \subseteq D\setminus C$$ is independent, we see that $C\cap C'\neq\emptyset$.
By the circuit exchange axiom (C3) there now exists a circuit $C''\subseteq C\cup C'$ in $M$. Let $A=\{C,C',C''\}$. Then we have
\[ n(\supp(A))=n(C\cup C')\leq n(F)\leq2<3=|A|, \]
hence $A\in\mathcal{A}_0$. In fact, $A$ is a circuit of $\delta M$, since by Lemma \ref{lemma: delta M is simple} there are no dependent sets of size $1$ or $2$ in $\delta M$.
\end{proof}

Indeed, the circuits $C, C'$ as constructed are the first two circuits in an \emph{ear decomposition} of $M$, as in \cite{coullard}. We only construct them explicitly for completeness.

\begin{remark}
By Lemma \ref{lemma: derived circuits}, every circuit in $\delta M$ can be written as $A_1\cup A_2\backslash\{C\}$ for some other circuits $A_1, A_2$ with $C\in A_1\cap A_2$.  Proposition \ref{proposition: delta M has triangle} shows that the converse of this statement does not hold: if $N$ is a matroid such that every circuit can be written as $C=(C_1 \cup C_2)\backslash \{e\}$, there does not necessarily exists some $M$ such that $\delta M=N$. A counterexample is given by letting $N$ be a uniform matroid of rank at least three: we can clearly write every circuit as $C=(C_1 \cup C_2)\backslash \{e\}$, but since $N$ does not contain a triangle, it can not be a combinatorial derived matroid.
\end{remark}

\section{Examples}\label{sec: examples}

In this section we will give several examples of combinatorial derived matroids and compare them to earlier definitions. 

\begin{example}
Consider the matroid $Q_6$ from~\cite[Example 7.2.4]{oxley2011matroid} whose geometric representation is given in Figure~\ref{fig:matroid Q6}.

\begin{figure}[!ht]
\centering
\begin{tikzpicture}[scale=.6]
\coordinate (1) at (4,4);
\coordinate (2) at (2,3);
\coordinate (3) at (0,2);
\coordinate (4) at (2,1);
\coordinate (5) at (4,0);
\coordinate (6) at (5,2);
\draw (1) -- (3) -- (5);
\draw[fill=black] (1) circle (.15);
\draw[fill=black] (2) circle (.15);
\draw[fill=black] (3) circle (.15);
\draw[fill=black] (4) circle (.15);
\draw[fill=black] (5) circle (.15);
\draw[fill=black] (6) circle (.15);
\node[above=2pt] at (1) {$1$};
\node[above=2pt] at (2) {$2$};
\node[left=2pt] at (3) {$3$};
\node[below=2pt] at (4) {$4$};
\node[below=2pt] at (5) {$5$};
\node[right=2pt] at (6) {$6$};
\end{tikzpicture}
    \caption{Geometric representation of matroid $Q_6$~ \cite[Figure 7.9]{oxley2011matroid}}
    \label{fig:matroid Q6}
\end{figure}
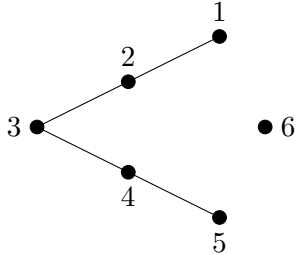
This matroid contains $11$ circuits: $123$, $345$ and all $4$-sets not containing the lines, e.g., $1245$. 

Let us construct $\mathcal{A}_0$. Direct computation shows that $|\supp(A)|\geq 5$ for all $A \in \mathcal{A}_0$. Hence, applying Lemma~\ref{lem: cardinality nondecreasing} we know that $|\supp(A)|\geq 5$ for all $A \in \mathcal{A}$. 

We will show how to construct the circuits of $\delta M$ with support of cardinality $5$ by fixing the excluded element of the support. First, fix the element $1$. Then any set $A \subseteq\{345, 2346, 2356, 2456\}$ such that $|A|\geq 3$ is contained in $\mathcal{A}_0$. Similar construction applies to $\{2,4,5\}$. Now fix $3$. Any set $A \subseteq\{1245, 1246, 1256, 1456, 2456\}$ such that $|A|\geq 3$ is contained in $\mathcal{A}_0$. Finally, there is one element of $\mathcal{A}_0$ that excludes $6$: $A=\{123,345,1245\}$.

\end{example}
\subsection{Uniform matroids}
\begin{example}\label{example: U(3,6)}
Let $M=U(3,6)$ be the rank $3$ uniform matroid on the ground set $[6]$. This is a representable matroid whose representations are $[6,3]$ MDS codes. 

 We give two representations $Q_1$ and $Q_2$ of $M$ over $\mathbb{F}_7$ below using generator matrices for the dual linear codes:
	\[G_{Q_1^\perp}=
	\begin{bmatrix}
	1 & 2 & 1 & 5 & 0 & 0\\
	1 & 5 & 0 & 0 & 5 & 1\\
	0 & 0 & 5 & 1 & 2 & 1
	\end{bmatrix}
	\]
	and 
	\[G_{Q_2^\perp}=
	\begin{bmatrix}
	1 & 1 & 1 & 1 & 0 & 0\\
	0 & 0 & 1 & 1 & 1 & 1\\
	0 & 1 & 0 & 1 & 3 & 4
	\end{bmatrix}
	.\]
Since this matroid is representable, we can apply the definition of derived matroids from~\cite{oxley2019dependencies}. The generator matrices of the derived matroids corresponding to $G_{Q_1^\perp}$ and $G_{Q_2^\perp}$ are:
\[G_{\delta M(Q_1)}=
	\begin{bmatrix}
	1 & 0 & 0 & 1 & 4 & 1 & 0 & 1 & 1 & 2 & 0 & 5 & 3 & 2 & 0\\
	2 & 0 & 4 & 5 & 0 & 0 & 4 & 1 & 6 & 0 & 4 & 6 & 2 & 0 & 4\\
	1 & 5 & 0 & 0 & 1 & 0 & 4 & 3 & 0 & 3 & 6 & 2 & 0 & 6 & 3\\
	5 & 1 & 2 & 0 & 0 & 1 & 0 & 0 & 4 & 4 & 6 & 0 & 4 & 6 & 4\\
	0 & 2 & 3 & 5 & 4 & 3 & 6 & 6 & 4 & 3 & 4 & 0 & 0 & 0 & 0\\
	0 & 1 & 3 & 1 & 6 & 6 & 1 & 0 & 0 & 0 & 0 & 2 & 2 & 2 & 5
	\end{bmatrix}
	\]
and
	\[G_{\delta M(Q_2)}=
	\begin{bmatrix}
    1 & 0 & 0 & 1 & 1 & 1 & 0 & 3 & 4 & 1 & 0 & 2 & 3 & 1 & 0\\
	1 & 0 & 1 & 1 & 0 & 0 & 1 & 4 & 5 & 0 & 1 & 3 & 4 & 0 & 6\\
	1 & 1 & 0 & 0 & 1 & 0 & 6 & 6 & 0 & 5 & 3 & 6 & 0 & 4 & 3\\
	1 & 1 & 1 & 0 & 0 & 6 & 0 & 0 & 1 & 4 & 4 & 0 & 1 & 3 & 2\\
	0 & 1 & 3 & 6 & 4 & 3 & 2 & 6 & 6 & 1 & 6 & 0 & 0 & 0 & 0\\
	0 & 1 & 4 & 6 & 3 & 2 & 3 & 0 & 0 & 0 & 0 & 1 & 1 & 6 & 6
	\end{bmatrix}
	.\]

Direct computation in \verb|Macaulay2| shows that both derived matroids have $751$ circuits, of which $712$ belong to the intersection of the two matroids. Let $\delta \mathcal{C}_i$ denote the set of circuits in the corresponding derived matroid. We want to study $\delta \mathcal{C}_1\triangle \delta \mathcal{C}_2$. It follows that $|\delta \mathcal{C}_i-\delta \mathcal{C}_j|=39$. Further computation shows that in both cases $|\delta \mathcal{C}_i-\delta \mathcal{C}_j|$ contains $3$ sets of cardinality $3$ and $36$ sets of cardinality $4$. The sets of cardinality $3$ are independent in the other derived matroid, whereas the sets of cardinality $4$ are dependent sets but not minimal in the other matroid. Specifically, the sets that are circuits in $\delta M(Q_1)$ but are independent in $\delta M(Q_2)$ are given by the columns
\[\{6,11,12\},\{3,8,14\},\{4,10,15\}\]
and
the sets that are circuits in $\delta M(Q_2)$ but are independent in $\delta M(Q_1)$ are given by the columns
\[\{1,2,4\},\{1,3,5\},\{2,8,13\}.\]
This shows that there is a bijection of sets between the two matroids, but they are not isomorphic, because in $\delta M(Q_2)$, the sets $\{1,2,4\},\{1,3,5\}$ have a non-empty intersection but all circuits of cardinality three in $\delta \mathcal{C}_1\triangle \delta \mathcal{C}_2$ are disjoint. So different representations of the same matroid can give non-isomorphic derived matroids in the sense of~\cite{oxley2019dependencies}.

Let us now compare these two represented derived matroids with the combinatorial derived matroid from Definition~\ref{def:combinatorial derived matroid}. Let $A \in \mathcal{A}_0$. For future record we note that the represented derived matroids $\delta M(Q_1)$ and $\delta M(Q_2)$ have $32256$ dependent sets. If $|A|\leq 2$, then $|A|\leq n(A)$ which contradicts~(\ref{eq:A0 definition}),  so $|A|\geq 3$. Since $M=U(3,6)$, $2 \leq n(A) \leq 3$. Therefore, any set $A \subset \mathcal{C}$ with $|A|\geq 4$ will be an element of $\mathcal{A}_0$. There are
\[\sum_{i=4}^{15}\binom{15}{i}=32192\]
such dependent sets in $\delta M$. If $|A|=3$, then to satisfy~(\ref{eq:A0 definition}) 
the nullity needs to be equal to $2$, so $|\operatorname{supp}(A)|=5$. In this case, there are $\binom{6}{5}=6$ choices of $\operatorname{supp}(A)$ and $\binom{5}{4}=5$ circuits contained in each support. Hence, there are a total of
\[\binom{6}{5}\cdot \binom{5}{3} =60\]
dependent sets $A \in \mathcal{A}_0$ of this type. 
Therefore, there are $32252$ elements in $\mathcal{A}_0$.
The set $\mathcal{A}_0$ is already closed under the operations $\epsilon$ and $\uparrow$ as we will show in Corollary~\ref{corollary: U(n-3,3)}, so $\mathcal{A}=\mathcal{A}_0$. We conclude that both $\delta(Q_1)$ and $\delta(Q_2)$ are not the combinatorial derived matroid, since they don't have the same number of dependent sets.

Next, we will find a representation of $M$ such that the derived and combinatorial derived matroid are the same. By Proposition~\ref{proposition: U(k,n)}, we get that $\delta U(3,6)=(\mathcal{C},\mathcal{A}_0)$, so Corollary~\ref{corollary:NumberOfDep} applies to $U(3,6)$. Thus, it suffices to find a representation $Q$ such that $\delta_{OW}(Q)$ has precisely $32252$ dependent sets, which is efficiently done by randomly constructed matrices over large enough fields.
 
The combinatorial derived matroid is representable over a sufficiently large field extension of $\mathbb{F}_7$. For example, the following is a generator matrix of $U(3,6)$ over $\mathbb{F}_{49}$, where $a$ is a primitive element over $\F_7$:
 \begin{equation}\label{matrix: U(3,6) over F_49}
    G_{Q^\perp}=
	\begin{bmatrix}
	-2& 2a+2& 3a+2& -a+2& -a-3& a+3\\
	3a+2& -3a-1& 2a+3& -1& -2a-1& -a\\
	-a+2 & 3& -2a+1& a& -2a+1& a+3
	\end{bmatrix} 
 \end{equation}
Direct computation confirms that its represented derived matroid has $32252$ dependent sets, and thus indeed equals the combinatorial derived matroid. 

Furthermore, the combinatorial derived matroid of $U(3,6)$ is representable over any field of characteristic $0$. For example, the following integral matrix has the reprentable derived matroid equal to the combinatorial derived matroid:
 \begin{equation}\label{matrix: U(3,6) over ZZ}
    G_{Q^\perp}=
	\begin{bmatrix}
	8& 8& 9& 6& 7& 8\\
	2& 7& 5& 7& 4&8\\
	7& 8& 0& 5& 6& 4
	\end{bmatrix}
 \end{equation}
 Both~(\ref{matrix: U(3,6) over F_49}) and~(\ref{matrix: U(3,6) over ZZ}) were generated as random matrices in the corresponding fields, which coincides with the intuitive interpretation of the combinatorial derived matroid being the most common matroid. 
\end{example}

\begin{proposition}\label{proposition: U(k,n)}
Let $M=U(k,n)$ be a uniform matroid. Then the combinatorial derived matroid $\delta M$ is generated by the set
\begin{equation*}\label{eq:A0 of U(k,n)}
  \mathcal{A}_0:=\{A \subseteq \mathcal{C}: |A|> |\supp(A)|-k\}.
\end{equation*}
Moreover, for all $A\in\mathcal{A}_0$ it holds that $k+2\leq|\supp(A)|\leq n$ and $|A|\geq 3$.
\end{proposition}

\begin{proof}
Since the collection of circuits of $U(k,n)$ is exactly the collection of sets of cardinality $k+1$, we have  $n(A)=|\supp(A)|-k$ for all $A\subseteq\mathcal{C}$ and thus $\mathcal{A}_0$ is of the described form. If $|A|=1$ then $|\supp(A)|=k+1$ so $A\notin\mathcal{A}_0$. If $|A|=2$ then $|\supp(A)|\geq k+2$ so also $A\notin\mathcal{A}_0$. Therefore $|A|\geq 3$ for all $A\in\mathcal{A}_0$ and this implies $k+2\leq|\supp(A)|\leq n$.
\end{proof}

In other words, $\mathcal{A}_0$ can be constructed by taking any subset $S \subseteq [n]$ such that $|S|\geq k+2$. Any set $A\subseteq \mathcal{C}$ such that $\supp (A)=S$ and $|A|> |S|-k$, will be an element of $\mathcal{A}_0$. In particular, any set $A \subset \mathcal{C}$ with $|A|\geq n-k+1$ will be an element of $\mathcal{A}_0$. 
The lower bound $|A|\geq 3$ from Proposition~\ref{proposition: U(k,n)} is sharp because $\delta U(k,n)$ contains a triangle, as shown in Proposition~\ref{proposition: delta M has triangle}.

We randomly generated matrices $R \in \mathbb{Z}^{k \times n}$. Such $R$ should be sufficiently generic, so with probability 1,  $M(R)=U(k,n)$. Furthermore, we conjecture that $\delta_{OW}M(R)=\delta U(k,n)$. Then in  Table~\ref{tab:circuits of delta U(k,n)} we present the empirical distribution of sizes of circuits of $\delta_{OW}M(R)$, which we expect to be indicative of the sizes of circuits of $\delta U(k,n)$.

\begin{table}[!ht]
\centering
{
\begin{tabular}{lllll}
        & \multicolumn{4}{l}{\textbf{size of circuits of $\delta M$}} \\
\textbf{Matroid }&\textbf{ 3 }            & \textbf{4 }             & \textbf{5}                & \textbf{6}               \\
$U(2,6)$  & 60            & 510            & 3432             &                 \\
$U(2,7)$  & 140           & 1785           & 24024            & 222600          \\
$U(3,5)$  & 10            &                &                  &                 \\
$U(3,6)$  & 60            & 735            &                  &                 \\
$U(3,7)$  & 210           & 5145           & 127232           &                
\end{tabular}
}
\caption{Number of circuits of given size in  $\delta_{OW}(M(R))$ for a random matrix $R \in \mathbb{Z}^{k \times n}$}
\label{tab:circuits of delta U(k,n)}
\end{table}

\begin{corollary}\label{corollary: U(n-3,3)}
Let $M=U(k,n)$ be a uniform matroid with $k\geq n-3$. Then $\delta M=(\mathcal{C}, \mathcal{A}_0)$. 
\end{corollary}

\begin{proof}
We will show that $\mathcal{A}_0$ as defined in Proposition \ref{proposition: U(k,n)} is invariant under the operations from Definition \ref{def:Operations}. The case $k\geq n-2$ is immediate, so assume $k=n-3$. 

Let $A_1 \in \mathcal{A}_0$ and $A_1\subseteq A_2 \subseteq \mathcal{C}$. If $\supp (A_2)= \supp (A_1)$, then $A_2 \in \mathcal{A}_0$ because $|A_2|\geq |A_1|> |\supp (A_1)|-(n-3)$. If $\supp (A_2)\supsetneq \supp (A_1)$, then 
\[|A_2|\geq |A_1|+1>|\supp(A_1)|+1-(n-3) \geq |\supp(A_2)|-(n-3)\]
and so $A_2\in \mathcal{A}_0$. This implies $\mathcal{A}_0={\uparrow}\mathcal{A}_0$.

Now take $A_1, A_2 \in \mathcal{A}_0$ such that $A_1 \cap A_2 \neq \emptyset$.
If $|A_1 \cup A_2|\geq 5$, then $|(A_1 \cup A_2) \backslash \{C\}|\geq 4$ for all $C \in A_1\cap A_2$ and so $(A_1 \cup A_2) \backslash \{C\} \in \mathcal{A}_0$. Assume now $|A_1|=|A_2|=3$ and $|A_1 \cup A_2|=4$. Then $|A_1 \cap A_2|=2$, so $A_1 \cap A_2 \notin \mathcal{A}_0$.

Any $A\in\mathcal{A}_0$ with $|A|=3$ has $|\supp(A)|=n-1$ so for any $C \in A$, $\supp(A)-\{C\}=\supp(A)$. Since $|A_1 \cap A_2|=2$ it follows that $\supp(A_1)=\supp(A_1 \cap A_2)=\supp(A_2)$ and hence $\supp((A_1\cup A_2)\backslash \{C\})=n-1$. This shows that derived dependent sets of cardinality 3 supported on $n$ elements of the ground set of $U(n-3,n)$ cannot be constructed via the operation $\epsilon$. So $\epsilon({\uparrow}\mathcal{A}_0)=\mathcal{A}_0$.

We conclude that $\mathcal{A}=\mathcal{A}_0$.
\end{proof}

The bound from Corollary~\ref{corollary: U(n-3,3)} is sharp. To see this, take $U(n-4,n)$. Let $A_1 \in \mathcal{A}_0$ and $A_1\subseteq A_2 \subseteq \mathcal{C}$. If $\supp (A_2)\supsetneq \supp (A_1)$, it is possible that~(\ref{eq:A0 definition}) is no longer satisfied, for example, if $|A_1|= |\supp (A_1)|-n+5$, $|A_2|= |A_1|+1$ but $|\supp (A_2)|\geq |\supp (A_1)|+2$. Hence, $\mathcal{A}_1\supsetneq \mathcal{A}_0$.

\subsection{Graphical matroids}
We will next give an example of the derived matroid of the graphical matroid $M(K_4)$. We believe that this may be an illuminating example for future studies of graphical matroids in general, and perhaps even more generally of binary matroids. A key feature of the example is that $\delta M(K_4)$ turns out to have strictly fewer dependent sets than the Longyear derived matroid of $M(K_4)$. This foreshadows some of our conjectures in Section 8. 

\begin{example}\label{ex:K4}
Let $M=M(K_4)$ be the graphical matroid with ground set $E=\{ab, ac, ad, bc, bd, cd\}$, labelled $1,\dots 6$ as in Figure \ref{fig:K4}, and circuits  $$\mathcal{C}=\{124, 135, 236, 456, 1346, 1256, 2345\}.$$

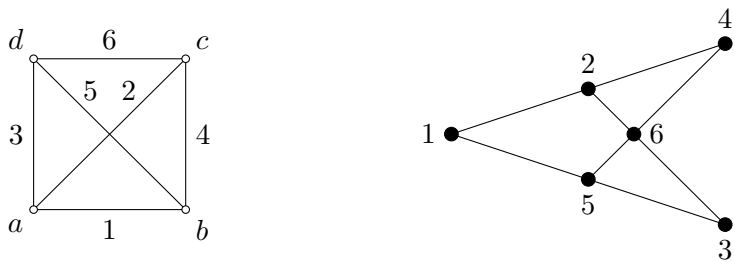
\begin{figure}[!ht]
\centering
\raisebox{-0.5\height}{ 
\begin{tikzpicture}[scale=.5]
\coordinate (a) at (0,0);
\coordinate (b) at (4,0);
\coordinate (c) at (4,4);
\coordinate (d) at (0,4);
\draw (a) -- node[below] {1} (b);
\draw (a) -- ++(1,1) -- node[above=2pt] {2} (c);
\draw (a) -- node[left] {3} (d);
\draw (b) -- node[right] {4} (c);
\draw (b) -- ++(-1,1) -- node[above=2pt] {5} (d);
\draw (c) -- node[above] {6} (d);
\draw[fill=white] (a) circle (.1);
\draw[fill=white] (b) circle (.1);
\draw[fill=white] (c) circle (.1);
\draw[fill=white] (d) circle (.1);
\node[anchor=north east] at (a) {$a$};
\node[anchor=north west] at (b) {$b$};
\node[anchor=south west] at (c) {$c$};
\node[anchor=south east] at (d) {$d$};
\end{tikzpicture} }
\hspace{2cm}
\raisebox{-0.5\height}{
\begin{tikzpicture}[scale=.6]
\coordinate (1) at (0,2);
\coordinate (2) at (3,3);
\coordinate (3) at (6,0);
\coordinate (4) at (6,4);
\coordinate (5) at (3,1);
\coordinate (6) at (4,2);
\draw (1) -- (4) -- (5);
\draw (1) -- (3) -- (2);
\draw[fill=black] (1) circle (.15);
\draw[fill=black] (2) circle (.15);
\draw[fill=black] (3) circle (.15);
\draw[fill=black] (4) circle (.15);
\draw[fill=black] (5) circle (.15);
\draw[fill=black] (6) circle (.15);
\node[left=2pt] at (1) {$1$};
\node[above=2pt] at (2) {$2$};
\node[below=2pt] at (3) {$3$};
\node[above=2pt] at (4) {$4$};
\node[below=2pt] at (5) {$5$};
\node[right=2pt] at (6) {$6$};
\end{tikzpicture} }
\caption{The graph $K_4$ (left) and the matroid $M(K_4)$ (right).}
\label{fig:K4}
\end{figure}

Note that these circuits correspond to the cycles $abc$, $abd$, $acd$, $bcd$, $abcd$, $abdc$, $acbd$ in the graph, respectively. Now $n(M)=3$ and $n(S)\leq 2$ for all $S\subsetneq E$. Hence, we have that $$\mathcal{A}_0=\{A\subseteq\mathcal{C}: |A|\geq 4\}\cup\{A\subseteq\mathcal{C}: |A|=3 \mbox{ and }\supp A\neq E\}.$$ We see that the three element sets in $\mathcal{A}_0$ are \begin{align*}\{124,135,2345\}, \, \{124,236,1346\}, \, \{124,456,1256\}, \\
\{135,236,1256\}, \, \{135,456,1346\}, \, \{236,456,2345\},&\end{align*} corresponding to the six ways to write a four-cycle in $K_4$ as the symmetric difference of two triangles. But these six three element lines form the non-trivial circuits of a matroid on ground set $\mathcal{C}$, namely the {\em non-Fano matroid $F_7^-$}~\cite[Example 1.5.13]{oxley2011matroid}, depicted in Figure \ref{fig:nonFano}. Thus we have $\mathcal{A}=\epsilon\mathcal{A}_0=\mathcal{A}_0$, so $\delta(M(K_4))$ is itself the non-Fano matroid. 
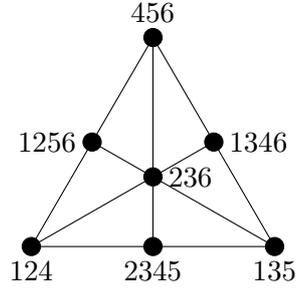
\begin{figure}
\centering
\begin{tikzpicture}[scale=.8]
\coordinate (1) at (0,0);
\coordinate (2) at (2,0);
\coordinate (3) at (4,0);
\coordinate (4) at (3,{sqrt(3)});
\coordinate (5) at (2,{2*sqrt(3)});
\coordinate (6) at (1,{sqrt(3)});
\coordinate (7) at (2, {2/sqrt(3)});
\draw (1) -- (3) -- (5) -- cycle;
\draw (1) -- (7) -- (4);
\draw (2) -- (5);
\draw (3) -- (6);
\draw[fill=black] (1) circle (.15);
\draw[fill=black] (2) circle (.15);
\draw[fill=black] (3) circle (.15);
\draw[fill=black] (4) circle (.15);
\draw[fill=black] (5) circle (.15);
\draw[fill=black] (6) circle (.15);
\draw[fill=black] (7) circle (.15);
\node[below=2pt] at (1) {$124$};
\node[below=2pt] at (2) {$2345$};
\node[below=2pt] at (3) {$135$};
\node[right=2pt] at (4) {$1346$};
\node[above=2pt] at (5) {$456$};
\node[left=2pt] at (6) {$1256$};
\node[right=2pt] at (7) {$236$};
\end{tikzpicture}
\caption{The non-Fano matroid $F_7^-$, labelled via the isomorphism to $\delta M$.}
\label{fig:nonFano}
\end{figure}

\end{example}

\subsection{Non-representable matroids}

In this section we consider the V\'{a}mos matroid. It is not representable, so its derived matroid in the sense of Oxley--Wang does not exist. Also, Cheung~\cite{cheung} proved that the V\'{a}mos matroid does not have an adjoint. Its combinatorial derived matroid, however, does exist. In our analysis we will particularly focus on studying the sets $$\Delta\mathcal{A}_{i}:=\epsilon\mathcal{A}_{i}-\mathcal{A}_{i}.$$

\begin{example}[V\'{a}mos matroid]\label{example: Vamos matroid}
Let $M$ be the V\'{a}mos matroid. We summarise our finding here and the detailed computations are located in Appendix~\ref{appendix: vamos}. V\'{a}mos matroid is a non-representable and non-algebraic matroid of rank $4$ on the ground set $[8]$. All sets of three or fewer elements are independent and among the $70$ sets of four elements only the $5$ depicted as gray rectangles in Figure~\ref{fig:Vamos matroid} are circuits. 

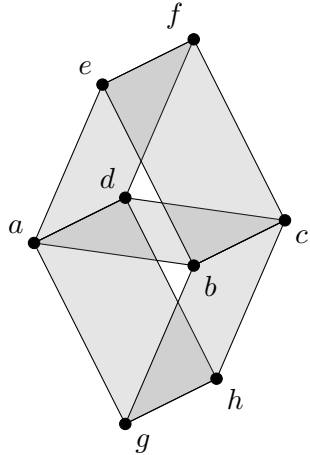
\begin{figure}[!ht]
\centering
\begin{tikzpicture}[scale=.3]
\coordinate (a) at (0,8);
\coordinate (b) at (7,7);
\coordinate (c) at (11,9);
\coordinate (d) at (4,10);
\coordinate (e) at (3,15);
\coordinate (f) at (7,17);
\coordinate (g) at (4,0);
\coordinate (h) at (8,2);
\filldraw[fill=gray, fill opacity=.2] (a) -- (b) -- (c) -- (d) -- cycle;
\filldraw[fill=gray, fill opacity=.2] (a) -- (d) -- (f) -- (e) -- cycle;
\filldraw[fill=gray, fill opacity=.2] (a) -- (d) -- (h) -- (g) -- cycle;
\filldraw[fill=gray, fill opacity=.2] (b) -- (c) -- (f) -- (e) -- cycle;
\filldraw[fill=gray, fill opacity=.2] (b) -- (c) -- (h) -- (g) -- cycle;
\draw[fill=black] (a) circle (.25);
\draw[fill=black] (b) circle (.25);
\draw[fill=black] (c) circle (.25);
\draw[fill=black] (d) circle (.25);
\draw[fill=black] (e) circle (.25);
\draw[fill=black] (f) circle (.25);
\draw[fill=black] (g) circle (.25);
\draw[fill=black] (h) circle (.25);
\node[anchor=south east] at (a) {$a$};
\node[anchor=north west] at (b) {$b$};
\node[anchor=north west] at (c) {$c$};
\node[anchor=south east] at (d) {$d$};
\node[anchor=south east] at (e) {$e$};
\node[anchor=south east] at (f) {$f$};
\node[anchor=north west] at (g) {$g$};
\node[anchor=north west] at (h) {$h$};
\end{tikzpicture}
\caption{V\'{a}mos matroid. The gray rectangles are the only four element circuits.}
\label{fig:Vamos matroid}
\end{figure}

The ground set of the derived matroid $\delta M$ has 41 elements. There are 290 3-sets and 14656 minimal 4-sets in $\mathcal{A}_0$. The former have support of cardinality $6$ and the latter have support of cardinality $7$. All sets of cardinality at least 5 are included in $\mathcal{A}_{0}$. 

Let us consider $\Delta\mathcal{A}_{0}$. It does not contain any $3$-sets. This is because if $|A|=3$ and $A \in \mathcal{A}_1-\mathcal{A}_0$, then $A=A_1 \cup A_2 \backslash \{C\}$ with $C \in A_1 \cap A_2$, $|A_1 \cap A_2|=2$, $A_1, A_2$ in $\mathcal{A}_0$ and $|A_1|=|A_2|=3$. For any distinct $C, C'$ in $\mathcal{C}$, $|C \cup C'|\geq6$. On the other hand, $|\supp A_1|=|\supp A_2|=6$, so $\supp A_1=\supp (A_1 \cap A_2)=\supp A_2=\supp (A_1 \cup A_2)$. Hence,
\[6 \leq |\supp A|\leq |\supp A_1 \cup A_2| =6,\]
that is, $A$ is a $3$-set supported on $6$ elements of $E$, so $A \in \mathcal{A}_0$. 
Let
\[\Delta\mathcal{A}_{03}:=\{A \in \Delta\mathcal{A}_{0}: A=A_1\cup A_2 \backslash \{C\}\text{ with } |A|=4,\ |A_1|=|A_2|=3\}.\]
It is shown in Appendix~\ref{appendix: vamos} that $\Delta\mathcal{A}_{0}=\Delta\mathcal{A}_{03}$,
so the only new elements generated with $\epsilon(\mathcal{A}_0)$ come from combinations of 3-sets. This is a consequence of~\eqref{eq:A0 definition} and does not propagate to subsequent iterations. For example, there exist new 4-sets in $\Delta \mathcal{A}_1$ that are generated by $A_1 \in \mathcal{A}_1$ and $A_2 \in \Delta \mathcal{A}_0$ with $3 \leq |A_1|\leq 4$ and $|A_2|=4$. The large number of 4-sets makes it computationally infeasible to directly check whether all of them are dependent.

We will return to the V\'{a}mos matroid in Example~\ref{example: Vamos hanging}, where we will find circuits in $\delta(M)$ that fail to satisfy some conditions that hold for circuits in Oxley-Wang derived matroids.
\end{example}

\section{Connectedness}\label{sec: connectedness}

As with many other notions about matroids, taking the combinatorial derived matroid commutes with taking the direct sum. That is: the combinatorial derived matroid of a direct sum of two matroids is the direct sum of the combinatorial derived matroid of these matroids. The proofs of the two results below are inspired by the proofs of Proposition 17 and Theorem 18 of \cite{oxley2019dependencies}.

\begin{proposition}
Let $M=M_1\oplus M_2$. Then $\delta(M)=\delta(M_1)\oplus \delta(M_2)$.
\end{proposition}
\begin{proof}
The circuits of $M$ are given by $\mathcal{C}(M)=\mathcal{C}(M_1)\cup\mathcal{C}(M_2)$: every circuit of $M$ is either a circuit of $M_1$ or a circuit of $M_2$. This means the ground set of $\delta M$ can be written as the disjoint union of $\mathcal{C}(M_1)$ and $\mathcal{C}(M_2)$.

First we show that $\mathcal{E}_0(M)=\mathcal{E}_0(M_1)\cup\mathcal{E}_0(M_2)$, that is, any set in $\mathcal{E}_0(M)$ consists of circuits of either only $M_1$, or only $M_2$. Let $A\in\mathcal{E}_0$. Then we can write $A=A_1\cup A_2$ with $A_1\subseteq\mathcal{C}(M_1)$, $A_2\subseteq\mathcal{C}(M_2)$ and $A_1\cap A_2=\emptyset$. Since $A\in\mathcal{A}_0$, we have that $|A|>n(A)$ with $n$ the nullity function of $M$. This implies
\begin{align*}
|A_1|+|A_2| &= |A| \\
 & > n(A) \\
 & = n(A_1\cup A_2) \\
 & = n(A_2\cup A_2)+n(A_1\cap A_2) \\
 & \geq n(A_1)+n(A_2).
\end{align*}
Since nullity is bounded by cardinality, $|A_1|+|A_2|>n(A_1)+n(A_2)$ implies that $|A_i|>n(A_i)$ for at least one of $i=1,2$. This means that $A_i\in\mathcal{A}_0(M)$. By minimality of $A$, we get that $A=A_i$.

Let $n_i$ be the nullity function of $M_i$, then $n(A_i)=n_i(A_i)$, so $|A_i|>n_i(A_i)$ and thus $A_i\in\mathcal{A}_0(M_i)$. In fact $A_i$ has to be in $\mathcal{E}_0(M_i)$, because otherwise it would not be minimal in $\mathcal{A}_0(M)$. This proves that $\mathcal{E}_0(M)=\mathcal{E}_0(M_1)\cup\mathcal{E}_0(M_2)$. 

We now construct the circuits of $\delta M$ as in Proposition \ref{proposition: construction of circuits}. The operation $\epsilon$ only constructs new sets from pairs of elements of $\mathcal{E}_j(M)$ that have nonempty intersection, and so for all $j\geq0$,
\[ \epsilon\left(\mathcal{E}_j(M_1)\sqcup \mathcal{E}_j(M_2)\right)=\epsilon(\mathcal{E}_j(M_1))\sqcup\epsilon( \mathcal{E}_j(M_2)). \]
This means that inductively we get $\mathcal{E}_j(M)=\mathcal{E}_j(M_1)\cup\mathcal{E}_j(M_2)$ for all $j$. We conclude that $\mathcal{E}(M)=\mathcal{E}(M_1)\cup\mathcal{E}(M_2)$ and thus $\delta(M)=\delta(M_1)\oplus \delta(M_2)$.
\end{proof}

In the previous proposition we have seen that if $M$ is not connected, then $\delta M$ is not connected. We will next prove the opposite: if $M$ is connected, then $\delta M$ is also connected. In order to do so, we first need a partial result about connectivity among the fundamental circuits of $M$ with respect to a fixed basis $B$.

\begin{proposition}\label{proposition: all fundamental circuits in the same component} Let $B$ be a basis of a connected matroid $M$, and let $e, f\in E-B$. Then the circuits $C_{eB}$ and $C_{fB}$ are in the same connected component of $\delta M$.
\end{proposition}

\begin{proof}
Assume first that $C_{eB}\cap C_{fB}\neq\emptyset$. Then $C_{eB}\cup C_{fB}\subseteq B\cup\{e,f\}$, so
$n(C_{eB}\cup C_{fB})\leq n(B\cup\{e,f\})=2$, and as $C_{eB}\cup C_{fB}$ contains two circuits, its nullity is indeed exactly two. By the circuit exchange axiom (C3), we have a third circuit $C'\subseteq C_{eB}\cup C_{fB}$, and so we have the dependent triangle $$\{C_{eB}, C_{fB}, C'\}\in \mathcal{A}_0\subseteq \mathcal{A}.$$ Since $\delta M$ is simple, this triangle is indeed a circuit in $\delta M$.  

Now consider the bipartite graph $G_B(M)$ with vertex classes $B\sqcup (E\setminus B)$ and edge set $$\{(i,j): i\in E\setminus B, j\in C_{iB}\}.$$ If there is a path of length $2$ from $e$ to $f$ in this graph, then $C_{eB}\cap C_{fB}\neq \emptyset$, so $C_{eB}$ and $C_{fB}$ are in the same connected component of $\delta M$. By transitivity of the property of being contained in a connected component (Proposition 4.1.2 in \cite{oxley2011matroid}), the same holds whenever there is a path of any (even) length from $e$ to $f$ in $G_B(M)$. 

But by Proposition 4.3.2 in \cite{oxley2011matroid}, the graph $G_B(M)$ is connected whenever the matroid $M$ is. Hence, all vertices in $E\setminus B$ are in the same connected component of the graph $G_B(M)$, and so all fundamental circuits $C_{eB}$ are in the same connected component of $\delta M$.
\end{proof}

\begin{theorem}
Let $M$ be a matroid on ground set $E$ such that $\supp(\delta M)=E$. Then $M$ is connected if and only if its combinatorial derived matroid $\delta M$ is connected.
\end{theorem}

\begin{proof}
Let $M$ be connected. By Proposition \ref{proposition: all fundamental circuits in the same component} all fundamental circuits with respect to some fixed basis $B$ of $M$ are elements of the same connected component of $\delta M$. We will now use the base exchange axiom iteratively to show that all fundamental circuits of $M$ are elements of the same connected component of $\delta M$, proving that $\delta M$ is connected.

If $M$ has only one basis, we are done. So, let $B_1$ and $B_2$ be bases of $M$. We will use induction on the size of $B_1 \Delta B_2$, which is always an even number as all bases have the same cardinality.

Assume now that whenever $|B_1 \Delta B_2|\leq k$, then all fundamental circuits with respect to $B_1$ or $B_2$ are elements of the same connected component of $\delta M$. Let $|B_1 \Delta B_2|= k+2$ and take $a \in B_1\backslash B_2$. By the basis exchange axiom, there exists $b \in B_2\backslash B_1$, such that $B'=(B_1 \backslash\{a\})\cup \{b\}$ is a basis of $M$. Let $C_{b}$ be the fundamental circuit of $b$ with respect to $B_1$ and $C_{a}$ be the fundamental circuit of $a$ with respect to $B'$. Then $C_a \subseteq B' \cup \{a\}=B_1 \cup \{b\}$. The uniqueness of fundamental circuits implies that $C_a=C_b$. This implies that all fundamental circuits with respect to $B'$ are elements of the same connected component of $\delta M$ as the fundamental circuits with respect to $B_1$. But now $B' \Delta B_2\subseteq (B_1 \Delta B_2) \backslash \{a,b\} $, so $|B' \Delta B_2|\leq k$ and we can apply the induction assumption and conclude that the fundamental circuits with respect to $B_2$ are elements of the same connected component as those with respect to $B'$, and hence also of same connected component as the fundamental circuits of $B_1$.

Applying the procedure above to all bases of $M$ gives that all fundamental circuits of $M$ are elements of the same connected component of $\delta M$. Since all circuits are fundamental circuits with respect to some basis, it follows that $\delta M$ has only one connected component and is thus connected.
\end{proof}

\section{Independent sets and rank}\label{sec: rank}

In this section we derive some further structural results about the combinatorial derived matroid. In cases where we could not find proofs or counterexamples, we make conjectures.

The following result was proven in \cite{MR4246979} for the Oxley--Wang derived matroid. It shows how to find independent sets in $\delta_{OW}$, independent of the representation of $M$.

\begin{lemma}[Lemma 5.4 of \cite{MR4246979}]\label{lemma: hanging elements}
Let $M$ be an $\mathbb{F}$-representable matroid on the ground set $[n]$. Take any set of circuits $S=\{C_1,\ldots,C_m\}$ such that every circuit $C_i$ satisfies
\[ C_i-\bigcup_{C_j\in S-\{C_j\}}C_j\neq\emptyset. \]
Then $S$ is independent in the derived matroid $\delta_{OW}$ of every $\mathbb{F}$-representation of $M$.
\end{lemma}

One can aim to prove a similar result for the combinatorial derived matroid, also for matroids that are not representable. A first step is the following lemma.

\begin{lemma}\label{lemma: hanging in A0}
Let $M$ be a matroid with $\mathcal{C}$ its set of circuits. Let $S\subseteq\mathcal{C}$ with $S\notin\mathcal{A}_0$ and let $C\in\mathcal{C}$ such that $C-\supp(S)\neq\emptyset$. Then $S\cup\{C\}\notin\mathcal{A}_0$.
\end{lemma}
\begin{proof}
Since $C\cap\supp(S)\subsetneq C$, the set $C\cap\supp(S)$ is independent in $M$ and hence $n(C\cap\supp(S))=0$. Because $S\notin\mathcal{A}_0$, we have $|S|\leq n(S)$. Combining this with the supermodularity of nullity, we get that
\begin{align*}
n(S\cup\{C\}) &=n(\supp(S)\cup C)+n(C\cap\supp(S)) \\
 & \geq n(S)+n(C) \\
 & \geq |S|+1 \\
 & =|S\cup\{C\}|
\end{align*}
and it follows that $S\cup\{C\}\notin\mathcal{A}_0$.
\end{proof}

A next step could be to prove a similar result for $\mathcal{A}_i$ instead of $\mathcal{A}_0$. Then, because $\mathcal{A}=\bigcup\mathcal{A}_i$, we would have the following.

\begin{conjecture}\label{conjecture: hanging elements}
Let $M$ be a matroid with $\mathcal{C}$ its set of circuits. Let $S\subseteq\mathcal{C}$ with $S\notin\mathcal{A}$ and let $C\in\mathcal{C}$ such that $C-\supp(S)\neq\emptyset$. Then $S\cup\{C\}\notin\mathcal{A}$.
\end{conjecture}

Note that from the above conjecture Lemma \ref{lemma: hanging elements} for the combinatorial derived matroid follows. Unfortunately, Conjecture \ref{conjecture: hanging elements} is not true, as the next example (due to Santiago G\'uzman Pro) shows.

\begin{example}\label{example: Vamos hanging}
Let $M$ be the V\'{a}mos matroid as in Example~\ref{example: Vamos matroid}. Let
\[A_1=\{adef, adgh, defgh\} \quad \text{and} \quad A_2=\{adef,adgh,bcef,bcgh\}.\]
Then $A_1 \in \mathcal{A}_0$ by the nullity condition and it was shown in Example~\ref{example: Vamos matroid} that $A_2 \in \mathcal{A}_1$. We now apply operation $\epsilon$ (Definition \ref{def:Operations}) on $A_1$ and $A_2$. Since $|A_1\cap A_2|=2$, $A_1\cap A_2\notin\mathcal{A}_1$ by Lemma \ref{lemma: delta M is simple}, and thus $S:=(A_1\cap A_2)\backslash\{adef\}\in\epsilon(\mathcal{A}_1)$. So we get $S=\{adgh,bcef,bcgh,defgh\}$ and $S\in\mathcal{A}$.

Let $C=adgh$ and consider $S\backslash\{C\}$: it needs to be either independent or dependent in $\delta M$. Suppose it is independent, so $S\backslash\{C\}\notin\mathcal{A}$. We have that $C-\supp(S\backslash\{C\})=a$. Applying Conjecture \ref{conjecture: hanging elements} now gives that $S\notin\mathcal{A}$, which is a contradiction. Now suppose $S\backslash\{C\}\in\mathcal{A}$ and let $C'=defgh$. We have that $C'-\supp(S\backslash\{C,C'\})=d$. Applying Conjecture \ref{conjecture: hanging elements} gives that $S\backslash\{C,C'\}\in\mathcal{A}$, which is a contradiction since by Lemma \ref{lemma: delta M is simple}, $S\backslash\{C,C'\}$ is independent because it has size $2$. We conclude that Conjecture \ref{conjecture: hanging elements} cannot be true.
\end{example}

Although this example does not disprove Lemma \ref{lemma: hanging elements} for the combinatorial derived matroid, it shows that new ideas are needed in order to prove or disprove the statement.

Unfortunately, we were not able to prove a more precise statement about the independent sets of $\delta M$. We do have the following straightforward observation about the rank of $\delta M$.

\begin{lemma}\label{lemma: rank derived}
Let $M$ be a non-empty connected matroid of rank $k$. Then the rank of the combinatorial derived matroid $\delta M$ is at most $n-k$.
\end{lemma}
\begin{proof}
The nullity of any subset of $\mathcal{C}$ is at most the nullity of $M$, which is $n-k$. This means that any subset of cardinality $n-k+1$ of $\mathcal{C}$ is in $\mathcal{A}_0$ and is thus dependent in $\delta M$. Now the rank of $\delta M$ is the cardinality of its largest independent set, which we just saw is at most $n-k$.
\end{proof}

It is unknown if the rank of the combinatorial derived matroid is always equal to $n-k$. We give a conjecture and a question towards a solution to this problem.

\begin{conjecture}
Let $M$ be a matroid such that $M^*$ has an adjoint. Then $\delta M$ is isomorphic to an adjoint of $M^*$. In particular, the rank of $\delta M$ is equal to $n-k$.
\end{conjecture}

It is shown in~\cite[Lemma 3]{oxley2019dependencies} that if $B$ is a basis of a representable matroid $M$ and $R$ any representation of $M$, then the set of fundamental circuits with respect to $B$ is a basis of $\delta_{OW}(R)$. This motivates the following question.
\begin{question}
    Does there always exist a basis $B$ of $M$ such that a subset of the fundamental circuits with respect to $B$ is a basis in $\delta M$?
\end{question}

\section{Comparison of definitions}\label{sec: comparison of def}
In this section, we present some partial results relating our construction to previous notions of derived matroids from the literature. In Section~\ref{sec:OxleyWangComp}, we relate $\delta M$ to the Oxley--Wang derived matroid $\delta_{OW}(R)$, where $R$ is a representation of $M$, and as a special case of this to the Longyear derived matroid $\delta_L(M)$ when $M$ is binary. In Section~\ref{sec:AdjointComp}, we compare the lattice of flat $\mathcal{F}(\delta(M))$ to adjoint lattices of the lattice of cyclic sets $\mathcal{U}(M)\cong\mathcal{F}(M^*)$. All results in this section are admittedly rather preliminary, and mainly serve as inspiration for further research. 

\subsection{Oxley and Wang}\label{sec:OxleyWangComp}
The relation between the combinatorial derived matroid and the Oxley--Wang derived matroid is not entirely straightforward. In particular, Example 5.5 shows that when $M=M(K_4)$ and $\F$ has characteristic two, there are representations of $M$ over $\F$, but none of these representations yields $\delta_{OW}(R)\cong \delta(M)$, even if we replace $\F$ by an extension field. However, we hope that in the precise sense of Conjecture~\ref{conjecture: CategoricallyGeneric}, the combinatorial derived matroid has less dependences than any Oxley-Wang derived matroid.

Our first results in this direction are the following. 

\begin{lemma}\label{lm:A0generic}
If $\delta M=(\mathcal{C},\mathcal{A}_0)$, then all dependent sets in $\delta M$ are dependent in $\delta_{OW} (Q)$ for every representation $Q$ of $M$.
\end{lemma}

\begin{proof}
Let $A\subseteq\mathcal{C}$.
For a representation $Q$ of $M$ and $C\in A$, denote by $q_C$ the circuit vector supported on $C$ as in Section \ref{sec: Oxley Wang construction}. If $A\in\mathcal{A}_0$, then we have \begin{equation*}
|A|>n(\supp A)=\dim(Q^\perp(\supp S))\geq\dim\operatorname{span}\{q_C : C\in A\}.
\end{equation*} Thus the vectors $\{q_C: C\in A\}$ are linearly dependent, so $A$ is dependent in $\delta_{OW}(Q)$.
\end{proof}

This shows that for matroids such that $\delta M=(\mathcal{C},\mathcal{A}_0)$, the combinatorial derived matroid is generic, in a precise sense. In particular, this applies to uniform matroids $U(k,n)$ with $k\geq n-3$, by Corollary~\ref{corollary: U(n-3,3)}. We also expect this to apply to many other naturally occurring matroids of high rank.
The simple Lemma~\ref{lm:A0generic} also shows that if $\delta M = (\mathcal{C},\mathcal{A}_0)$ we can characterize representations for which the Oxley--Wang derived matroid equals the combinatorial derived matroid from mere enumerative properties of $\delta_{OW}(Q)$.

\begin{corollary}\label{corollary:NumberOfDep}
If $\delta M=(\mathcal{C},\mathcal{A}_0)$ and $Q$ is a representation of $M$ such that $\delta_{OW}(Q)$ has the same number of dependent sets as $\delta M$, then $\delta_{OW}(Q)=\delta M$.
\end{corollary}
\begin{proof}
By Lemma~\ref{lm:A0generic} we have $\mathcal{D}(\delta M)=\mathcal{A}_0\subseteq\mathcal{D}(\delta_{OW}(Q))$, so since the two sets are finite, they must be equal if they have the same cardinality. Since $\delta M$ and $\delta_{OW}(Q)$ also have the same ground set $\mathcal{C}$, they must thus be equal.
\end{proof}

In general, if $M$ is representable, we conjecture that the combinatorial derived matroid is generic in two different ways; one categorical and one algebraic-geometric. In order to state Conjectures~\ref{conjecture: CategoricallyGeneric} and \ref{conjecture: GeometricallyGeneric}, we need some auxiliary standard definitions.

\begin{definition}
Let $E$ be a finite set. The {\em weak order} on the class $\mathcal{M}$ of matroids with ground set $E$ is the partial order with $M_1\geq M_2$ if every dependent set in $M_1$ is also dependent in $M_2$. 
\end{definition}

\begin{conjecture}\label{conjecture: CategoricallyGeneric}
Let M be a representable matroid with set of circuits $\mathcal{C}$. Then $\delta M\geq\delta_{OW}(R)$ for every representation $R$ of $M$, where the order relation is the weak order on the set of matroids on $\mathcal{C}$.
\end{conjecture}

Note that Lemma~\ref{lm:A0generic} proves Conjecture~\ref{conjecture: CategoricallyGeneric} in the case that $\delta M=(\mathcal{C},\mathcal{A}_0)$. It is conceivable that Conjecture~\ref{conjecture: CategoricallyGeneric} may be approached with techniques similar to those in~\cite{jackson2021}. There, the authors studied conditions under which the collection a matroids with some circuits prescribed have a maximal element in the weak order. 

The next example shows that it is possible that $M$ is representable over a field $\F$, without any of the representations of $M$ over $\F$ (or indeed over any extension field of $\F$) yielding $\delta M$ as an Oxley-Wang derived matroid.

\begin{example}
Consider the graphical matroid $M=M(K_4)$ as in Example~\ref{ex:K4}. Recall that $\delta M$ is the non-Fano matroid $F_7^-$, with dependent sets $\mathcal{A}=\mathcal{A}_0$.

We will now look at Oxley--Wang derived matroids of $M(K_4)$. By Lemma~\ref{lm:A0generic}, all sets in $\mathcal{A}=\mathcal{A}_0$ are dependent in $\delta_{OW}(R)$ for any representation over any field. Moreover, a triple of circuit vectors can only be linearly dependent if every element in $E$ occurs in at least two of the circuits, which follows from Lemma~\ref{lemma: hanging in A0}.
The only possible dependent set in $\delta_{OW}(R)$, except for the sets in $\mathcal{A}$, is therefore the triple of $4$-circuits $\{1346, 1256, 2345\}$. It is straightforward to see that these circuits are linearly dependent in the (projectively unique) representation of $M$ over a field $\F$ or characteristic~$2$, but not in a representation over fields of characteristic $\neq 2$. This can intuitively be understood if we interpret a linear combination of the circuit vectors supported on the circuits $\{1346, 1256, 2345\}$ as a formal sum of the cycles $\{abcd, abdc, acbd\}$ in the graph. In order for such a formal sum to take the value zero on the edge $6=cd$, the coefficient in front of the (directed) cycles $abcd$ and $abdc$ must be the same, say $\alpha$. But then the formal sum of cycles traverses the edge $ab$ $2\alpha$ times {\em in the same direction}, and so it can only be zero if $2\alpha=0$, {\em i.e.,} if $\mathrm{char}(\F)=2$. 
    
Thus, we see that $\delta_{OW}(R)$ is the Fano matroid in characteristic $2$, and the non-Fano matroid in characteristic $\neq 2$. In particular, we get $\delta(M)=\delta_{OW}(R)$ for a representation $R$ of $M$ if and only if the representation is in characteristic $\neq 2$. \end{example}

For algebraic-geometric notions of genericity, fix a field $\F$ and let $\overline{\F}$ be its algebraic closure. Assume that $M$ is representable over $\F$, with $|E(M)|=n$ and $r(M)=k$. We can then define the variety $V(M)=V_\F(M)\subseteq \overline{\F}^{k\times n}$ of matrices that represent $M$. This is clearly an algebraic variety, since it is defined by the vanishing of certain finitely many minors and the non-vanishing of finitely many others. We also have a variety $V(N)\subseteq \overline{\F}^{(n-k)\times |\mathcal{C}|}$ for every possible $\overline{\F}$-representable matroid $N$ on the ground set $\mathcal{C}=\mathcal{C}(M)$. Then the map $$R\mapsto \delta(R)$$ subdivides $V(M)$ into components $$V_N(M)=\{R\in V(M) : \delta_{OW}(R)=N\},$$ for different possible Oxley--Wang derived matroids $N$ of $M$. From this perspective, genericity is captured by the following conjecture.

\begin{conjecture}\label{conjecture: GeometricallyGeneric}
For every matroid $N\neq \delta(M)$ on $\mathcal{C}$, if $V_N(M)$ is non-empty, then it has positive codimension in $V(M)$.
\end{conjecture}

Note that it is possible that $V_N(M)$ is empty over certain ground fields $\F$. Example 5.5 shows that this is true, in particular, when $M=M(K_4)$ and $\F$ is the two element field.

Since there are only finitely many possible Oxley--Wang derived matroids, this would imply that $V_{\delta(M)}(M)$ would be a full-dimensional subvariety of $V(M)$. In particular, assume that $\F$ is a finite field and that $\F'/\F$ is a finite field extension whose degree $[\F':\F]$ tends to infinity. Conjecture~\ref{conjecture: GeometricallyGeneric} would then imply that asymptotically almost all representations $R$ of $M$ over $\F'$ would satisfy $\delta_{OW}(R)=\delta(M)$.

\subsection{Adjoint}\label{sec:AdjointComp}
The construction of the combinatorial derived matroid is iterative. If a set of circuits of $M$ satisfies the nullity condition \eqref{eq:A0 definition} in Definition \ref{def:combinatorial derived matroid}, the set is in $\mathcal{A}_0$ hence dependent. If not, the only way to check its (in)dependence is by constructing iteratively all dependent sets of $\delta M$ and see if our set is one of them.

The fact that there are sets for which (in)dependence is hard to check, means it is also difficult to define a closure function of $\delta M$. We are not optimistic about finding a definition of closure in $\delta M$ without some form of iteration.

As a result of the lack of closure in $\delta M$, we have not managed to compare the definition of $\delta M$, that uses dependent sets, to that of the adjoint (Section \ref{sec:adjoint}), that uses closed sets. We expect the following to be true.

\begin{conjecture}
If an adjoint of a matroid $M$ exists, one of them is equal to the combinatorial derived matroid of the dual. Moreover, every adjoint of $M$ is less $\delta M^*$ in the weak order on matroids with ground set $\mathcal{C}$.
\end{conjecture}

\section{Suggestions for further research}\label{sec: further research}

We have developed the notion of the combinatorial derived matroid. Contrary to earlier constructions for a derived matroid, our construction exist for all matroids and is unique.

There are many open questions about our construction. Apart from the questions and conjectures already stated in the text, we list here some more ideas for further research.

\begin{question}
Is there a characterization of the class of matroids $N$ that are isomorphic to the combinatorial derived matroid $\delta M$ of some matroid $M$? 
By Proposition 4.12, a necessary condition is that every element in $N$ that is contained in a circuit is also contained in a circuit of size $3$. Many more necessary conditions like this should probably first be obtained, before a complete characterization is feasible.
\end{question}

\begin{question}
For which matroids does it hold that $\delta M=M$ or $\delta M=M^*$? The latter property may be more natural, since the rank of $\delta M$ is always equal to the rank of $M^*$. It is also possible to repeat the procedure of taking the combinatorial derived matroid of the dual multiple times, i.e. study the sequence $$M, \delta(M^*), \delta(\delta(M^*)^*),\dots.$$
Does this sequence converge? For earlier constructions of the derived matroid, these questions have been studied in~\cite[Section 3]{oxley2019dependencies} and~\cite[Section 4]{jurrius2015coset}, see also~\cite{CrapoBlog}. 
\end{question}

\begin{question}
Is there any relation between $\delta M$ and $\delta M^*$?
\end{question}

\begin{question}
There are matroids that do not have an adjoint, whereas the combinatorial derived matroid always exists. This means one of the properties in Definition~\ref{definition: adjoint} and/or Lemma~\ref{lemma: adjoint properties} is not in general satisfied by the combinatorial derived matroid. Which one?
\end{question}

\begin{question}
If $X$ is a representation of $M$ by topological spheres, as in \cite{swartz2002toprep, anderson2009toprep}, does $X$ induce a representation $\delta X$ of $\delta M$ by topological spheres?
\end{question}

\section*{Acknowledgments} Kuznetsova was partially supported by the Academy of Finland Grant 323416. The authors would like to thank the anonymous reviewer for their careful reading of the manuscript and may useful suggestions for improvement. Furthermore we thank Santiago G\'uzman Pro for providing Example \ref{example: Vamos hanging}, a counterexample to a statement in an earlier version of this paper.

\bibliographystyle{plain}
\bibliography{references}

\begin{thebibliography}{10}

\bibitem{Alfter}
Marion Alfter, Walter Kern, and Alfred Wanka.
\newblock On adjoints and dual matroids.
\newblock {\em Journal of Combinatorial Theory, Series B}, 50:208--213, 1990.

\bibitem{anderson2009toprep}
Laura Anderson.
\newblock Homotopy sphere representations for matroids.
\newblock {\em Annals of Combinatorics}, pages 189--202, 2012.

\bibitem{cheung}
Alan~L Cheung.
\newblock Adjoints of a geometry.
\newblock {\em Canadian Mathematical Bulletin}, 17(3):363--365, 1974.

\bibitem{coullard}
Collette Coullard and Lisa Hellerstein.
\newblock Independence and port oracles for matroids, with an application to
  computational learning theory.
\newblock {\em Combinatorica}, 16, 02 1998.

\bibitem{crapo:2009}
Henry Crapo.
\newblock An algebra of pieces of space --- {H}ermann {G}rassmann to {G}ian
  {C}arlo {R}ota.
\newblock In Ernesto Damiani, Ottavio D'Antona, Vincenzo Marra, and Fabrizio
  Palombi, editors, {\em From Combinatorics to Philosophy}, pages 61--90,
  Boston, MA, 2009. Springer US.

\bibitem{Crapo}
Henry Crapo and William Schmitt.
\newblock The {W}hitney algebra of a matroid.
\newblock {\em Journal of Combinatorial Theory, Series A}, 91:215--263, 2000.

\bibitem{MR4246979}
Ragnar Freij-Hollanti and Olga Kuznetsova.
\newblock Information hiding using matroid theory.
\newblock {\em Adv. in Appl. Math.}, 129:102205, 18, 2021.

\bibitem{jackson2021}
Bill Jackson and Shin-ichi Tanigawa.
\newblock Maximal matroids in weak order posets, 2021.

\bibitem{jurrius2015coset}
Relinde Jurrius and Ruud Pellikaan.
\newblock The coset leader and list weight enumerator.
\newblock {\em Contemporary Mathematics}, 632:229--251, 2015.

\bibitem{CrapoBlog}
Joseph Kung.
\newblock The geometry of cocircuits and circuits. a tombeau for {H}enry
  {C}rapo 1933--2019.
\newblock Blog post, 2020.

\bibitem{longyear1980circuit}
Judith~Q Longyear.
\newblock The circuit basis in binary matroids.
\newblock {\em Journal of Number Theory}, 12(1):71--76, 1980.

\bibitem{oxley2011matroid}
James Oxley.
\newblock {\em Matroid theory}, volume~21 of {\em Oxford Graduate Texts in
  Mathematics}.
\newblock Oxford University Press, Oxford, second edition, 2011.

\bibitem{Oxley1996}
James Oxley, Dirk Vertigan, and Geoff Whittle.
\newblock On inequivalent representations of matroids over finite fields.
\newblock {\em Journal of Combinatorial Theory, Series B}, 67:325--343, 1996.

\bibitem{oxley2019dependencies}
James Oxley and Suijie Wang.
\newblock Dependencies among dependencies in matroids.
\newblock {\em The Electronic Journal of Combinatorics}, page P3.46, 2019.

\bibitem{Rota}
Gian-Carlo Rota.
\newblock Bowdoin college summer 1971 {NSF} conference on combinatorics, 1971.

\bibitem{swartz2002toprep}
Edward Swartz.
\newblock Topological representations of matroids.
\newblock {\em Journal of the American Mathematical Society}, pages 427--442,
  2002.

\end{thebibliography}

\newpage
\appendix
\section{Vamos matroid}\label{appendix: vamos}
Let $M$ be the V\'{a}mos matroid. It is a non-representable and non-algebraic matroid of rank $4$ on the ground set $[8]$. All sets of three or fewer elements are independent and among the $70$ sets of four elements only the $5$ depicted as gray rectangles in Figure~\ref{fig:Vamos matroid} are circuits. 

\smallskip
\textbf{The ground set of $\delta M$.} First we calculate the number of elements of the ground set of $\delta M$, that is, the number of circuits in $M$. There are 5 circuits of size 4 in $M$, and of the $\binom{8}{5}=56$ sets of size 5 there are $5\cdot4=20$ that contain a circuit of size 4. This gives $(56-20)+5=41$ circuits in $M$, hence $\delta M$ is a matroid on $41$ elements.

\smallskip
\textbf{The elements of $\mathcal{A}_0$.} To construct $\mathcal{A}_0$, first note that all subsets of $\mathcal{C}$ of size $0$, $1$ and $2$ are not contained in $\mathcal{A}_0$ by Lemma \ref{lemma: delta M is simple}. Moreover, all sets of size at least $5$ are in $\mathcal{A}_0$, since $n(M)=4$. For a $3$-set $A\subseteq\mathcal{C}$ to be contained in $\mathcal{A}_0$, we need that $n(A)\leq 2$, so $|\supp A| =6$. There are 290 such sets in $\mathcal{A}_0$. 

A 4-set $A$ is in $\mathcal{A}_0$ if and only if its support is of size at most 7. Motivated by Proposition~\ref{proposition: construction of circuits}, we would like to understand the minimal elements of cardinality 4 in $\mathcal{A}_0$. If $|\supp(A)|=6$, any 3-set subset $A'$ of $A$ satisfies $|\supp(A')|=6$ because any 3-set of circuits has support at least 6. But then $A'$ is a 3-set on support of size 6 and rank 4, so $A'$ is in $\mathcal{A}_0$. Hence, to include the minimal 4-sets of $\mathcal{A}_0$, we can require that 
$|\supp(A)|=7$, so $n(A)=3$ and $|A|-n(A)=1$. There are 14656 minimal 4-sets in $\mathcal{A}_0$.

\smallskip
\textbf{The sets $\mathcal{A}_0$ vs $\mathcal{A}$.} 
We will show that $\mathcal{A}\neq \mathcal{A}_0$. Let $A_1=\{adef,bcef,$ $abcd\}$ and $A_2=\{adgh,bcgh,abcd\}$. Then $n(A_1)=n(A_2)=6-4=2$ and $|A_1|=|A_2|=3$, so $A_1 \in \mathcal{A}_0$ and $A_2 \in \mathcal{A}_0$. Consider the set $A=(A_1 \cup A_2)-(A_1 \cap A_2)$. We have $n(A)=8-4=4$ and $|A|=4$, so $A \notin \mathcal{A}_0$. The set $A_1 \cap A_2$ is a singleton, so $A_1 \cap A_2 \notin \mathcal{A}_0$. Hence, $A \in \mathcal{A}_1 \subseteq \mathcal{A}$.

\smallskip
\textbf{The elements of $\Delta\mathcal{A}_{0}$.} All elements of $\Delta\mathcal{A}_{0}$ have cardinality 4. This is because all sets of cardinality at least 5 are included in $\mathcal{A}_{0}$. On the other hand, the set $\Delta\mathcal{A}_{0}$ does not contain any $3$-sets. This is because if $|A|=3$ and $A \in \mathcal{A}_1-\mathcal{A}_0$, then $A=A_1 \cup A_2 \backslash \{C\}$ with $C \in A_1 \cap A_2$, $|A_1 \cap A_2|=2$, $A_1, A_2$ in $\mathcal{A}_0$ and $|A_1|=|A_2|=3$. For any distinct $C, C'$ in $\mathcal{C}$, $|C \cup C'|\geq6$. On the other hand, $|\supp A_1|=|\supp A_2|=6$, so $\supp A_1=\supp (A_1 \cap A_2)=\supp A_2=\supp (A_1 \cup A_2)$. Hence,
\[6 \leq |\supp A|\leq |\supp A_1 \cup A_2| =6,\]
that is, $A$ is a $3$-set supported on $6$ elements of $E$, so $A \in \mathcal{A}_0$.

Next we analyse the 4-sets inside $\Delta\mathcal{A}_{0}$. First, consider 
\[\Delta\mathcal{A}_{03}:=\{A \in \Delta\mathcal{A}_{0}: A=A_1\cup A_2 \backslash \{C\}\text{ with } |A|=4,\ |A_1|=|A_2|=3\} .\]
All $A \in \Delta\mathcal{A}_{03}$ have cardinality $4$ and $|\Delta\mathcal{A}_{03}|=5949$. Of those, 373 have support of size 8, 5416 have support of size 7 and 160 sets have support of size 6. A further check of confirms that all of the elements in $\Delta\mathcal{A}_{03}$ that have support of size 8 are circuits in $\delta M$. The remaining sets have support of size at most 7, so they are contained in $\mathcal{A}_{0}$. 

The elements of $\Delta\mathcal{A}_{03}$ that have support of size 6 cannot be minimal because every 3-subset of such 4-sets is contained in $\mathcal{A}_0$. The supports avoiding the ``edges'' $ab$ and $cd$ and the ``diagonals'' $ac$ and $bd$ of the circuit $abcd$, appear in 15 elements of $\Delta\mathcal{A}_{03}$ each. The supports avoiding $ad$, $bc$, $ef$ and $gh$ do not appear in $\Delta\mathcal{A}_{03}$. For each remaining 6-subset $S$ of $E$ there are 5 elements of $\Delta\mathcal{A}_{03}$ supported on $S$. 

Similar symmetries happen among the elements of $\Delta\mathcal{A}_{03}$ that have support of size 7. The elements of $\{a,b,c,d\}$ are avoided 847 times each whereas the elements of $\{e,f,g,h\}$ are avoided 507 times each. 

Let us now consider the set
\[\Delta\mathcal{A}_{034}:=\{A \in \Delta\mathcal{A}_{0}: A=A_1\cup A_2 \backslash \{C\}\text{ with } |A|=4,\ |A_1|=3, \ |A_2|=4\} .\]
We claim that $\Delta\mathcal{A}_{034}=\emptyset$. This is because the assumption that $|A_1|=3$, $|A_2|=4$ and $|A|=4$ implies that $|A_1\cap A_2|=2$. Furthermore, since $A_1 \in \mathcal{A}_0$, $\supp A_1= \supp (A_1\cap A_2)\subseteq \supp A_2 $, so $\supp A =\supp A_2$ and $n(A)=n(A_2)$, so $A \in \mathcal{A}_0$.

Next, consider
\[\Delta\mathcal{A}_{04}:=\{A \in \Delta\mathcal{A}_{0}: A=A_1\cup A_2 \backslash \{C\}\text{ with } |A|=|A_1|=|A_2|=4\} .\]
Again, we claim that $\Delta\mathcal{A}_{04}=\emptyset$. The assumption $|A|=|A_1|=|A_2|$ implies that $|A_1 \cap A_2|=3$. If $|\supp (A_1 \cap A_2)|=6$ then $(A_1 \cap A_2) \in\mathcal{A}_0$, so $A \notin \Delta\mathcal{A}_{0}$. Then because $A_1, A_2 \in \mathcal{A}_0$, $7=\supp A_1=\supp (A_1 \cap A_2)=\supp A_2=\supp (A_1 \cup A_2)$. This means that $|\supp A|=7$, so $A \in \mathcal{A}_0$. 

\smallskip
\textbf{3-sets and 4-sets in $\mathcal{A}$.} Direct computation shows that the properties that $\Delta\mathcal{A}_{034}=\emptyset$ and  $\Delta\mathcal{A}_{04}=\emptyset$ depend on condition~\eqref{eq:A0 definition}. In fact, one can already generate new 4-sets in $\Delta \mathcal{A}_1$ by taking $A_1 \in \mathcal{A}_1$ and $A_2 \in \Delta \mathcal{A}_0$ with $3 \leq |A_1|\leq 4$ and $|A_2|=4$. 

At the same time, if a 3-set $A$ is an element of $\mathcal{A}$, then $A \in \mathcal{A}_0$. This is because if $A \in \mathcal{A}-\mathcal{A}_0$, then Lemma~\ref{lemma: delta M is simple} implies that $A=(A_1 \cup A_2)\backslash \{C\} $ for some 3-sets $A_1$ and $A_2$, so if $A \in \mathcal{A}_i-\mathcal{A}_{i-1}$ for some $i$, then  $A \in \Delta \mathcal{A}_{i-1}$ and $A_1,A_2 \in \mathcal{A}_{i-1}$. We will prove that $A \in \mathcal{A}_0$ by induction on $i$. We have already shown that $A \notin \Delta \mathcal{A}_0$. So let us assume that if a 3-set $A' \in \mathcal{A}_{i-1}$, then $A'\in \mathcal{A}_0$. Since $|A|=|A_1|=|A_2|$, implies that $|A_1\cap A_2|=2$. By the induction assumption, $A_1, A_2 \in \mathcal{A}_0$, so $\supp A_1=A_1\cap A_2=\supp A_2$ and has cardinality 6. Therefore, $|\supp A|=6$ and $A \in \mathcal{A}_0$.

\end{document}